\documentclass{amsart}

\usepackage{url,amssymb,enumerate,colonequals}

\usepackage{geometry}
\usepackage{mathrsfs} 
\usepackage[section]{placeins}
\usepackage{MnSymbol}
\usepackage{extarrows}
\usepackage{lscape}
\usepackage{circledsteps}
\usepackage[all,cmtip]{xy}

\usepackage[OT2,T1]{fontenc}
\usepackage{color}
\usepackage{tikzit}

\tikzstyle{white with black border}=[fill=white, draw=black, shape=circle]
\tikzstyle{black with white font}=[white, fill=black, draw=black, shape=circle, font={\small}, inner sep=0pt]
\tikzstyle{black simple style}=[fill=black, draw=black, shape=circle]
\tikzstyle{A7 node}=[fill={rgb,255: red,220; green,47; blue,2}, draw=black, shape=circle]
\tikzstyle{D5 node}=[fill={rgb,255: red,212; green,73; blue,6}, draw=black, shape=circle]
\tikzstyle{D4 node}=[fill={rgb,255: red,228; green,107; blue,6}, draw=black, shape=circle]
\tikzstyle{A3 node}=[fill={rgb,255: red,250; green,163; blue,7}, draw=black, shape=circle]
\tikzstyle{A2 node}=[fill={rgb,255: red,255; green,186; blue,8}, draw=black, shape=circle]
\tikzstyle{A1 node}=[fill={rgb,255: red,244; green,228; blue,9}, draw=black, shape=circle]

\tikzstyle{arrow}=[->]
\tikzstyle{new edge style 0}=[-, fill=none]
\tikzstyle{A7 edge}=[-, fill={rgb,255: red,220; green,47; blue,2}, draw={rgb,255: red,220; green,47; blue,2}]
\tikzstyle{D5 edge}=[-, fill={rgb,255: red,212; green,73; blue,6}, draw={rgb,255: red,212; green,73; blue,6}]
\tikzstyle{D4 edge}=[-, fill={rgb,255: red,228; green,107; blue,6}, draw={rgb,255: red,228; green,107; blue,6}]
\tikzstyle{A3 edge}=[-, fill={rgb,255: red,250; green,163; blue,7}, draw={rgb,255: red,250; green,163; blue,7}]
\tikzstyle{A2 edge}=[-, fill={rgb,255: red,255; green,186; blue,8}, draw={rgb,255: red,255; green,186; blue,8}]
\tikzstyle{dashed arrow}=[->, dashed]
\tikzstyle{A1 edge}=[-, fill={rgb,255: red,244; green,228; blue,9}, draw={rgb,255: red,244; green,228; blue,9}]
\tikzstyle{A3 edge white}=[-, fill={rgb,255: red,250; green,163; blue,7}, draw=white]

\usepackage{tikz-cd}
\usepackage{resizegather}
\usepackage[
        colorlinks=true, citecolor=darkestred, linkcolor=darkred, urlcolor=darkred,
        backref=page,
        pdfauthor={Alvaro Gonzalez Hernandez},
]{hyperref}

\usepackage{comment}
\usepackage{multirow}
\usepackage{amssymb}
\usepackage{amsfonts}
\usepackage{amsmath}
\usepackage[nobysame,alphabetic]{amsrefs}

\usepackage{holtpolt}
\usepackage{xcolor}
\usepackage{adjustbox}
\usepackage{colortbl}
\usepackage{multicol}
\usepackage{appendix}
\usepackage{graphicx}
\usepackage{wrapfig}
\usepackage{stackrel}
\usepackage{changepage} 
\usepackage{listings}
\usepackage{xcolor}
\usepackage{float}
\usepackage{tikz-cd} 
\usepackage{algorithm}
\usepackage{algpseudocode}
\usepackage{algorithmicx}
\usepackage{algpseudocode}
\usepackage{tcolorbox}
\tcbuselibrary{breakable, skins}
\usepackage{amsthm}
\usepackage{mdframed}
\usepackage{algorithmicx}
\usepackage{algpseudocode}

\newtheorem{alg}{Algorithm}

\newmdenv[
  linewidth=0.6pt,
  topline=true,
  bottomline=true,
  leftline=false,
  rightline=false,
  skipabove=\baselineskip,
  skipbelow=\baselineskip
]{algobox}

{%
  \begin{algobox}
  \begin{alg}[#1]\normalfont
}%
{%
  \end{alg}
  \end{algobox}
}

\newcommand{\dashedarrow}{\dashrightarrow}

\newcommand{\Q}{\mathbb{Q}}

\newcommand{\Z}{\mathbb{Z}}

\newcommand{\calA}{\mathcal{A}}

\newcommand{\calC}{\mathcal{C}}

\newcommand{\calJ}{\mathcal{J}}

\newcommand{\calO}{\mathcal{O}}

\newcommand{\calV}{\mathcal{V}}

\newcommand{\calX}{\mathcal{X}}

\DeclareMathOperator{\Gal}{Gal}

\DeclareMathOperator{\im}{im}

\DeclareMathOperator{\Jac}{Jac}
\DeclareMathOperator{\Jacc}{\mathcal{J}}
\DeclareMathOperator{\Kum}{Kum}

\DeclareMathOperator{\Pic}{Pic}

\DeclareMathOperator{\Sym}{Sym}

\newcommand{\BlO}{\operatorname{Bl}_{E_O}(Y)}

\newcommand{\Lone}{\mathcal{L}(\Theta_++\Theta_-)}
\newcommand{\Ltwo}{\mathcal{L}(2(\Theta_++\Theta_-))}

\numberwithin{equation}{section}

\newtheorem{theorem}{Theorem}[section]

\newtheorem{proposition}{Proposition}[section]
\theoremstyle{definition}

\theoremstyle{remark}

\definecolor{darkestblue}{HTML}{03045E}
\definecolor{darkblue}{HTML}{0077B6}
\definecolor{darkestred}{HTML}{c42706}
\definecolor{darkred}{HTML}{d44906}
\definecolor{lighterblue}{HTML}{00B4D8}

\DeclareRobustCommand{\SkipTocEntry}[5]{}

\makeatletter
\newcommand{\xdashrightarrow}[2][]{\ext@arrow 0359\rightarrowfill@@{#1}{#2}}
\newcommand{\xdashleftarrow}[2][]{\ext@arrow 3095\leftarrowfill@@{#1}{#2}}
\newcommand{\xdashleftrightarrow}[2][]{\ext@arrow 3359\leftrightarrowfill@@{#1}{#2}}
\def\rightarrowfill@@{\arrowfill@@\relax\relbar\rightarrow}
\def\leftarrowfill@@{\arrowfill@@\leftarrow\relbar\relax}
\def\leftrightarrowfill@@{\arrowfill@@\leftarrow\relbar\rightarrow}
\def\arrowfill@@#1#2#3#4{%
  $\m@th\thickmuskip0mu\medmuskip\thickmuskip\thinmuskip\thickmuskip
   \relax#4#1
   \xleaders\hbox{$#4#2$}\hfill
   #3$%
}
\makeatother
\begin{document}

\title{Explicit desingularisation of Kummer surfaces in characteristic two via specialisation}

\author{Alvaro Gonzalez-Hernandez}

\address{Mathematics Institute\\
    University of Warwick\\
    CV4 7AL \\
    United Kingdom\\}

\email{
\href{mailto:alvaro.gohe@outlook.com}{alvaro.gohe@outlook.com} 
}

\thanks{Website: \url{https://alvarogohe.github.io}}

\begin{abstract}
We study the birational geometry of the Kummer surfaces associated to the Jacobian  varieties of genus two curves, with a particular focus on fields of characteristic two. In order to do so, we explicitly compute a projective embedding of the Jacobian of a general genus two curve and, from this, we construct its associated Kummer surface.
This explicit construction produces a model for desingularised Kummer surfaces over any field of characteristic not two, and specialising these equations to characteristic two provides a model of a partial desingularisation.
Adapting the classic description of the Picard lattice in terms of tropes, we also describe how to explicitly find completely desingularised models of Kummer surfaces whenever the $p$-rank is not zero. In the final section of this paper, we compute an example of a Kummer surface with everywhere good reduction over a quadratic number field, and draw connections between the models we computed and a criterion that determines when a Kummer surface has good reduction at two.
\end{abstract}

\maketitle

\section{Introduction}
Kummer surfaces are quotients of abelian surfaces by the involution that sends any point to its inverse with respect to the group law on the surface. They are one of the most classical examples of K3 surfaces, and they have been studied extensively, due to their remarkable number of singular points and their deep connections with the geometry of abelian varieties \cite{Shioda1977OnSurfaces}.\\

In this article, we study the geometry of Kummer surfaces associated with Jacobians of genus two curves. In this setting, one can always obtain explicit equations for the Kummer surface, realised as a singular quartic surface in~$\mathbb{P}^3$. If the characteristic of the field of definition is not two, the associated quartic model has sixteen nodes corresponding to the $2$-torsion points \cite{Cassels1996Prolegomena2}. An explicit model of the desingularisation is known: it arises as the intersection of three quadrics in $\mathbb P^5$, and this model has connections with the computation of explicit equations of the Jacobian variety as the intersection of $72$ quadrics inside $\mathbb{P}^{15}$ (Section \ref{sectionsp}).\\

Outside the realm of geometry, one motivation to study Kummer surfaces comes from cryptography. In addition to elliptic curve cryptography, there has been a recent interest in studying cryptographic protocols that involve abelian varieties of higher dimension, for instance, in isogeny-based protocols like the Supersingular Isogeny Diffie-Hellman (SIDH) \cite{Clingher2025KummerFunctions, Costello2018ComputingSurfaces}. However, many of these protocols do not work in characteristic two, as in that case, Kummer surfaces behave quite differently \cite{Katsura1978On2}. This is because in characteristic two the $2$-torsion of the abelian surface is a subgroup of $(\Z/2\Z)^2$, and its associated Kummer surface therefore has fewer singular points, but of higher complexity (Section \ref{char2}).\\

As in the characteristic zero case, there is a way to construct an explicit model for the Kummer surface associated to the Jacobian of a genus two curve as a quartic in $\mathbb{P}^3$ \cite{Muller2010ExplicitCharacteristic}. However, following this construction does not generate a smooth model of the desingularisation of this quartic as the intersection of three quadrics in $\mathbb{P}^5$. In principle, this could suggest that over a field of characteristic two, Jacobians of genus two curves and Kummer surfaces behave completely differently compared to how they behave over a field of any other characteristic.\\

Our main result demonstrates that, despite these differences, the classical constructions can be adapted to work in characteristic two.\newpage

\begin{theorem}
Let $\calC$ be a genus-two curve over a perfect field $k$ of characteristic two. Then, its Jacobian variety $\Jacc$ admits an embedding as the intersection of $72$ quadrics in $\mathbb{P}^{15}$ and we provide explicit equations for this embedding. Furthermore, we computed a projective embedding for a partial desingularisation of the Kummer surface associated to $\Jacc$ as the complete intersection of three quadrics in $\mathbb{P}^5$.
\end{theorem}

These two results are made precise in Theorems \ref{th_kummer_in_char_2} and \ref{embeddjac}. As we shall see, both embeddings arise naturally by specialising suitable models from characteristic zero (Section \ref{compsection}).\\
 
 Recently, Katsura and Kondō \cite{Katsura2023KummerTwo} used the theory of quadric line complexes to also describe equations for partial desingularisations of Kummer surfaces as the intersection of quadrics in $\mathbb{P}^5$. Through different methods, we extend their results by showing that simpler models for these equations can always be computed over the field of definition of the curve (Section~\ref{PS}).\\
 
 In characteristic zero, there are sixteen special curves called tropes going through the singular points of a Kummer surface, and we will see how the specialisation of these curves provides a natural way to study the desingularisation of Kummer surfaces in characteristic two (Section \ref{sectionweddle}).\\

This general theme of studying Kummer surfaces in positive characteristic from the reduction of a model in characteristic zero will play an even bigger role in Section \ref{section good red} of the paper.\\

There is a well-known result of Fontaine \cite{Fontaine1985IlZ} (see also Abrashkin \cite{Abrashkin1988GaloisVectors}) which asserts that there does not exist any abelian scheme over $\Z$ and, as a consequence of this, there cannot exist abelian varieties defined over $\Q$ with everywhere good reduction. In a similar fashion, a lesser-known result, also due to Abrashkin \cite{Abrashkin1990ModularConjecture} and Fontaine \cite{Fontaine1991SchemasEntiers} independently, shows that there cannot exist K3 surfaces defined over the rationals that have everywhere good reduction.\\

Since Tate provided in the late sixties one of the first examples of elliptic curves with good reduction everywhere, the curve $E\,/\,\mathbb{Q}(\sqrt{29})$ defined as
\begin{align*}
\mathcal{E}: y^2+xy+\big(\tfrac{5+\sqrt{29}}{2}\big)^2y=x^3,
\end{align*}
many different techniques and methods have been developed in order to find elliptic curves with everywhere good reduction over number fields. In the case of abelian surfaces, it is relevant the work of Dembélé and Kumar \cite{Dembele2016ExamplesReduction}, Dembélé \cite{Dembele2021OnReduction}, and Dąbrowski and Sadek \cite{Dabrowski2021GenusFields} who all found explicit examples defined over quadratic fields of genus two curves whose Jacobians have everywhere good reduction over a quadratic number field.\\

After seeing that the question has a positive answer for abelian surfaces, one would naturally ask if it is then possible to find examples of K3 surfaces with everywhere good reduction over a number field. In Section \ref{section good red}, we will give a positive answer to this question by proving the following:

\begin{theorem}
There exist examples of Kummer surfaces with everywhere good reduction over a number field. 
\end{theorem}

More concretely, we construct an explicit example of a Kummer surface with everywhere good reduction over a quadratic field. To check that the Kummer surface has good reduction at all primes, we apply a criterion of Lazda and Skorobogatov \cite{Lazda2023ReductionCase} to study the reduction at two of an abelian surface with good reduction at all places, which is defined over a quadratic number field. This criterion involves studying the action of the absolute Galois group of the base field on the $2$-torsion points, and sheds light on the conditions that have to be met for a smooth model of a Kummer surface to also reduce to a smooth surface modulo two.\\

This paper comes with code that supports all the calculations and allows us to compute explicit equations for all the varieties that have been described. The latest updates of this code can be found \href{https://github.com/AlvaroGohe/Kummer-surfaces-and-Jacobians-of-genus-2-curves-in-characteristic-2}{here}, whereas the version of the code at the time of publication can be found \href{https://github.com/AlvaroGohe/Kummer-surfaces-and-Jacobians-of-genus-2-curves-in-characteristic-2/releases}{here}.\newpage

\section{Projective models of Kummer surfaces in characteristic not two} \label{sectionsp}
The theory of how to obtain explicit equations of a Kummer surface and its desingularisation over a field $k$ of characteristic zero was first described by Grant in the case of genus two curves with a rational branch point \cite{Grant1990FormalTwo} and Cassels and Flynn in a more general case \cite{Cassels1996Prolegomena2}. The following presentation adapts the method described by Flynn, Testa and Van Luijk \cite{Flynn2012Two-coverings2} to the case where we have a hyperelliptic curve described by a model of the form $y^2+g(x)y=f(x)$.\\

Let $k$ be a field of characteristic not equal to two, $k^s$ a separable closure of $k$ and $f(x)=\sum_{i=0}^6f_ix^i$ and $g(x)=\sum_{i=0}^3g_ix^i\in k[x]$ such that $\tilde{f}(x)=f(x)+\tfrac{1}{4}g(x)^2$ is a separable polynomial of degree six. We will denote by $\Omega$ the set of the six roots of $\tilde{f}$ in $k^s$, so that $k(\Omega)$ is the splitting field of $f$ over $k$ in $k^s$.\\

Let $\mathcal{C}$ be the smooth projective curve of genus two over $k$ associated with the affine curve in $\mathbb{A}^2_{x,y}$ given by $y^2+g(x)y=f(x)$, let $\Jacc$ denote the Jacobian of $\mathcal{C}$ and let $\Jacc[2]$ be its $2$-torsion subgroup.
All $2$-torsion points are defined over $k(\Omega)$, so $\Jacc[2](k(\Omega)) = \Jacc[2](k^s)$. We will denote by $W \subset \mathcal{C}$
the set of Weierstrass points of $\mathcal{C}$, corresponding to the set $\{(\omega_i, -\tfrac{1}{2}g(\omega_i)) : \omega_i\in \Omega, i\in\{1,\dots,6\}\}$ of points on the
affine curve.\\

Let $\iota$ denote the automorphism of $\Jacc$ defined by sending every point to its inverse with respect to the group law and let $K_\mathcal{C}$ be the canonical divisor of $\mathcal{C}$ that is supported at the points at infinity, that is, $K_\mathcal{C}=(\infty_+)+(\infty_-)$, where $\infty_+$ and $\infty_-$ are the two points at infinity, which may not be defined over the ground field individually.
For any $w \in W$, the divisor $2(w)$ is linearly
equivalent to $K_\mathcal{C}$ and $\sum_{w\in W}(w)$ is linearly equivalent to $3K_\mathcal{C}$. We let $\iota_h$ denote the hyperelliptic involution on $\mathcal{C}$ that sends $(x,y)$ to $(x,-y-g(x))$. We then have that $\iota_h(\infty_\pm)=\infty_\mp$.\\

For any point $P$ on $\mathcal{C}$ the divisor $(P)+(\iota_h(P))$ is linearly equivalent to $K_\mathcal{C}$, and there is a morphism $\mathcal{C}\times \mathcal{C}\rightarrow\Jacc$ sending $(P_1, P_2)$ to the divisor class $(P_1) + (P_2)-K_\mathcal{C}$, which
factors through the symmetric product of a curve with itself $\mathcal{C}^{(2)}$. The induced map $\mathcal{C}^{(2)}\rightarrow \Jacc$ is birational and each
nonzero element $D_{ij}$ of $\Jacc[2](k^s)$ is represented by
\begin{align*}D_{ij}=(w_i)- (w_j) \sim (w_j)-(w_i) \sim (w_i) + (w_j)-K_\mathcal{C}\end{align*}
for a unique unordered pair $\{w_i, w_j\}$ of distinct Weierstrass points.\\

We will denote by $P_O$ and $P_{ij}$ the image in $X$ of the identity of the group law and $D_{ij}$ respectively under the quotient map. Note that $P_{ij}$ lies in the field extension $k(\omega_i+\omega_j,\omega_i\omega_j)$. In fact, the map $\mathcal{C}^{(2)}\rightarrow \Jacc$ is the blow-up of $\Jacc$ at the origin $O$ of $\Jacc$. The inverse image of $O$ is the curve on $\mathcal{C}^{(2)}$ that consists of all the pairs $\{P,\iota_h(P)\}$. We may therefore identify the function field $k(\Jacc)$ of $\Jacc$ with that of $\mathcal{C}^{(2)}$
which consists of the functions in the function field 
\begin{align*}
k(\mathcal{C}\times \mathcal{C})=k(x_1,x_2)[y_1,y_2]/(y_1^2+g(x_1)y_1-f(x_1),y_2^2+g(x_2)y_2-f(x_2))
\end{align*}
which are invariant under the exchange of the indices. It is easy to check that for any two points $P_1$ and $P_2$ on $\mathcal{C}$ we have 
\begin{align*}
(P_1)+(P_2)-K_\mathcal{C}\sim -(\iota_h(P_1)+\iota_h(P_2)-K_\mathcal{C})
\end{align*} and $\iota$ on $\Jacc$ is induced by the involution $\iota_h$. We can then check that the induced automorphism $\iota^*$ of $k(\Jacc)$ fixes $x_1$ and $x_2$, and changes $y_1$ and $y_2$ by $-y_1-g(x_1)$ and $-y_2-g(x_2)$ respectively. For any function $h\in k(\Jacc)$ we say that $h$ is \textbf{even} or \textbf{odd} if $\iota^*(h)=h$ or $\iota^*(h)=-h$ respectively.\\

We will denote by $X$ the \textbf{Kummer surface} of $\Jacc$, $X=\Jacc/\langle\iota\rangle$, and by $Y$ the desingularised Kummer surface, that is, the blow-up of $X$ at the image of the fixed points of $\iota$. We will denote by $E_{ij}$ the $(-2)$-curve on $Y$ above the singular point $P_{ij}$ of $X$. Let $\Jacc'$ be the blow-up of $\Jacc$ in its $2$-torsion points.
We denote the $(-1)$-curve on $\Jacc'$ above the point $D_{ij}\in\Jacc[2]$ by $F_{ij}$. The involution $\iota$ on $\Jacc$ lifts to an involution on $\Jacc'$ and the resulting quotient is isomorphic to $Y$. \newpage Therefore, there is a morphism $\Jacc'\rightarrow Y$ with ramification divisor $\sum_{D_{ij}\in\Jacc[2]}F_{ij}$ that makes the following diagram commutative:
\[\begin{tikzcd}
	{\Jacc'} & \Jacc \\
	Y & X
	\arrow[from=1-2, to=2-2]
	\arrow[from=1-1, to=2-1]
	\arrow[from=1-1, to=1-2]
	\arrow[from=2-1, to=2-2]
\end{tikzcd}\]

For any Weierstrass point $w\in W$ of $\mathcal{C}$ we define $\Theta_w$ to be the divisor on $\Jacc$ that is the image of the divisor $\mathcal{C}\times\{w\}$ on $\mathcal{C}^{(2)}$, that is, $\Theta_w$ consists of all divisor classes represented by $(P)-(w)$ for some point $P\in\mathcal{C}$. These $\Theta_w$ are known as \textbf{theta divisors} and their doubles are all linearly equivalent. We then have the following result:
\begin{proposition}[\cite{Flynn2012Two-coverings2}] Suppose $w\in W$ is a Weierstrass point defined over $k$. The linear system $\lvert 2\Theta_w\rvert$ induces a morphism of $\Jacc$ to $\mathbb{P}^3_k$ that is the composition of the quotient map $\Jacc\rightarrow X$ and a closed embedding of $X$ into $\mathbb{P}^3_k$. The linear systems $\lvert 3\Theta_w\rvert$ and $\lvert 4\Theta_w\rvert$ induce closed embeddings of $\Jacc$ into $\mathbb{P}^8_k$ and $\mathbb{P}^{15}_k$ respectively.    
\end{proposition}
For any divisor $D$ on $\Jacc$, let $\mathcal{L}(D)=H^0(\Jacc,\mathcal{O}_{\Jacc}(D))$ and let $\ell(D)$ be its dimension. Let $\Theta_{+}$ and $\Theta_{-}$ be the images of the divisors $\mathcal{C}\times\{\infty_+\}$ and $\mathcal{C}\times\{\infty_-\}$, respectively, in $\Jacc$.\\

Then, $\Theta_++\Theta_-\sim 2\Theta_w$, so the maps induced by $\lvert 2\Theta_w\rvert$ and $\lvert 4\Theta_w\rvert$ can always be defined over the ground field, and the closed embeddings of $X$ and $\Jacc$ are described by the elements in the bases of $\mathcal{L}(\Theta_++\Theta_-)$ and $\mathcal{L}(2(\Theta_++\Theta_-))$.\\

A standard computation shows that $\ell(\Theta_++\Theta_-)=4$ and $\ell(2(\Theta_++\Theta_-))=16$. Moreover, one can explicitly construct four even functions $k_1,\dots,k_4$ and six odd functions $b_1,\dots,b_6$ in $k(\Jacc)$ such that:
\begin{itemize}
    \item The set $\{k_1,\dots,k_4\}$ forms a basis for $\mathcal{L}(\Theta_++\Theta_-)$ and therefore the linear system defines an embedding $\varphi_{\lvert \Theta_++\Theta_-\rvert}:\Jacc\hookrightarrow\mathbb{P}^3$, whose image is $X$. Furthermore, $X$ can be described as a quartic in $\mathbb{P}^3$ which has sixteen $A_1$ singularities, all defined over $k(\Omega)$.
    
    \item If we let $k_{ij}=k_ik_j$, the elements $\{k_{11},\dots,k_{44},b_1,\dots,b_6\}$ form a basis for $\mathcal{L}(2(\Theta_++\Theta_-))$ and therefore define an embedding $\varphi_{\lvert 2(\Theta_++\Theta_-)\rvert}:\Jacc\hookrightarrow\mathbb{P}^{15}$.\\
\end{itemize}
Let $\mathrm{Sym}^{2}\mathcal{L}(\Theta_++\Theta_-)$ denote the symmetry product of $\mathcal{L}(\Theta_++\Theta_-)$ with itself. Then the map
\begin{align*}
\mathrm{Sym}^{2}\mathcal{L}(\Theta_++\Theta_-)&\longrightarrow\mathcal{L}(2(\Theta_++\Theta_-))\\
k_i\ast k_j&\longmapsto k_{ij}
\end{align*}
is injective as the $\{k_{ij}\}$ are linearly independent. We can therefore identify $\mathrm{Sym}^{2}\mathcal{L}(\Theta_++\Theta_-)$ with its image in $\mathcal{L}(2(\Theta_++\Theta_-))$, and use this fact to find a projective embedding of $Y$, the desingularisation of $X$:
\begin{proposition}[\cite{Flynn2012Two-coverings2}]
There are direct sum decompositions
\begin{align*}
\mathcal{L}(2(\Theta_++\Theta_-))\quad&= \quad\langle \text{even coordinates}\rangle & &\oplus &&\langle \text{odd coordinates}\rangle\\
&=\quad \mathrm{Sym}^{2}\mathcal{L}(\Theta_++\Theta_-)&&\oplus & &\mathcal{L}(2(\Theta_++\Theta_-))(-\Jacc[2])\\
&=\quad H^0(X,\varphi_{\lvert\Theta_++\Theta_-\rvert}^*\mathcal{O}_{\mathbb{P}^3}(2))&&\oplus & & H^0(\Jacc',\mathcal{O}_{\Jacc'}(2(\Theta_++\Theta_-)-\sum F_{ij}))
\end{align*}
where $\mathcal{L}(2(\Theta_++\Theta_-))(-\Jacc[2])$ is the subspace of $\mathcal{L}(2(\Theta_++\Theta_-))$ of sections vanishing on the $2$-torsion points. Furthermore, the projection of $\Jacc\subset\mathbb{P}^{15}$ away from the even coordinates determines a rational map
\begin{align*}
\Jacc&\xdashrightarrow{\hspace{7mm}}\mathbb{P}^5\\
D&\longmapsto[b_1(D):\dots:b_6(D)]
\end{align*}
which induces the morphism $\Jacc'\rightarrow\mathbb{P}^5$ associated to the linear system $\lvert 4\Theta_w-\sum F_{ij}\rvert$ on $\Jacc'$, and factors as the quotient map $\Jacc'\rightarrow Y$ and a closed embedding $Y\hookrightarrow\mathbb{P}^5$.
\end{proposition}
The even coordinates are the ones given by the functions $\{k_{ij}\}_{1\leq i,j\leq4}$ and the odd coordinates are those given by $\{b_{i}\}_{1\leq i\leq6}$.\newpage As it was mentioned earlier, this basis defines an embedding of $\Jacc$ in $\mathbb{P}^{15}$ generated by $72$ quadrics:

\begin{itemize}
    \item A $20$-dimensional subspace of the space generated by these quadrics is spanned by the equations of the form $k_{ij}k_{rs} = k_{ir}k_{js}$ for $1\leq i, j, r,s\leq4$.
    \item An additional relation between the $k_{ij}$ comes from the fact that there is a relation between $\{k_1,\dots, k_4\}$ of degree four which defines the embedding of the Kummer surface in $\mathbb{P}^3$.
    \item The $21$ relations arise from the fact that the space of quadrics of $\{b_1,\dots,b_6\}$ has dimension $21$ and the product of two elements of $\mathcal{L}(2(\Theta_++\Theta_-))(-\Jacc[2])$ is an even function inside $\mathcal{L}(4(\Theta_++\Theta_-))$ and, therefore, it can be expressed as a polynomial of degree four on the $k_i$. From these relations we can explicitly construct an explicit birational map $X\dashedarrow Y$, defined outside of the singular locus of $X$, whose inverse $Y\dashedarrow  X$ is the blow-up of the sixteen singular points in $X$.
    \item Finally, it can be checked that there are eight relations between the elements of the form $b_ik_{j}$ with $1\leq i\leq6$, $1\leq j\leq4$. Multiplying each of these relations by $k_1$, $k_2$, $k_3$ and $k_4$, we obtain $32$ relations between the elements of the form $b_ik_{jr}$. Not all these relations are linearly independent, but they generate a $30$-dimensional space.\\
\end{itemize}

\subsection{Translation by a 2-torsion point}
Given any non-zero element $D_{ij}\in\Jacc[2]$, we can define an automorphism $\tau_{ij}$ on $\mathcal{J}$ by sending
\begin{align*}
\Jacc&\longrightarrow\Jacc\\
D&\longmapsto D+D_{ij}
\end{align*}
Then, the actions that $\tau_{ij}$ induces on $\Lone$ and $\Ltwo$ are linear \cite{Flynn1993The2} and, since the involution $\iota$ commutes with $\tau_{ij}$, we deduce that $\tau_{ij}$ induces a linear map in both $X$ and $Y$, which is defined over the field of definition of $P_{ij}$, which is $k(\omega_i+\omega_j,\omega_i\omega_j)$. 
Therefore, we have an action of $\Jacc[2]$ on both $X$ and $Y$ defined over $k$, and over $k(\Omega)$, this is an action of $(\Z/2\Z)^4$.\\

\subsection{Tropes of a Kummer surface}
It is known classically that the Kummer surface $X$ contains sixteen conics, known as \textbf{tropes}, which satisfy that each trope passes through six singular points, and through every singular point there are exactly six tropes.\\

In the case where the Kummer surface arises from the Jacobian of a genus two curve, we have a nice combinatorial description of these tropes in terms of subsets of the Weierstrass points:
\begin{itemize}
    \item There are six tropes of the form $T_i$ corresponding to the partitions of the set $\{1,\dots,6\}$ into sets of one and five elements of the form $\{\{i\},\{j,k,l,m,n\}\}$.
    The trope $T_i$ is defined to be the one going through the singular points $\{O,P_{ij},P_{ik}, P_{il},P_{im}, P_{in}\}$ and in the model of $X$ as a quartic in $\mathbb{P}^3$ that we have described, $T_i$ can be defined over the field extension $k(\omega_i)$.
    
    \item There are ten tropes of the form $T_{ijk}$ corresponding to the partitions of the set $\{1,\dots,6\}$ into two subsets of three elements $\{\{i,j,k\},\{l,m,n\}\}$. In this case, $T_{ijk}=T_{lmn}$ and the trope $T_{ijk}$ goes through the six singular points $\{P_{ij},P_{ik},P_{jk}, P_{lm}, P_{ln}, P_{mn}\}$. This trope is defined over the minimal field extension that is generated by the sums and products of $\{\omega_i,\omega_j,\omega_k, \omega_l, \omega_m,\omega_n\}$ which are invariant under the action of the permutations $(ijk)(lmn)$ and $(il)(jm)(kn)$.\\
\end{itemize}
Consider the subvariety $\calC\times\{w_i\}$ inside $\calC^{(2)}$. Then, another way of describing $T_i$ is as the image of $\calC\times\{w_i\}$ under the composition of the map $\calC^{(2)}\rightarrow\Jacc$ and the quotient $\Jacc\rightarrow X$.
The rest of the tropes can be obtained as the images of any of these tropes by a suitable translation by a $2$-torsion point, according to the rules:
\begin{align*}
\tau_{ij}(T_i)&=T_j, & \tau_{ij}(T_{ijk})&=T_k, &
\tau_{ij}(T_k)&=T_{ijk}, & \tau_{ij}(T_{ikl})&=T_{jkl},
\end{align*}
where we are assuming that $i,j,k,l$ are all different indices. According to how the polynomial $f(x)+\tfrac{1}{4}g(x)^2$ decomposes into irreducible polynomials over $k$, the number of tropes and singular points of $X$ defined over $k$ are described in table \ref{t1}.\newpage

\begin{table}[h]
\centering
\begin{tabular}{|
>{\columncolor[HTML]{FFFFFF}}c| ccc|}
\hline
\cellcolor[HTML]{FFFFFF}Partition & \cellcolor[HTML]{FFFFFF}$\#$ tropes of type $T_i$ & \cellcolor[HTML]{FFFFFF}$\#$ tropes of type $T_{ijk}$ & \cellcolor[HTML]{FFFFFF}$\#$ singular points \\ \hline
$\{1,1,1,1,1,1\}$                                                & 6                                                     & 10                                                        & 16                                           \\ 
$\{1,1,1,1,2\}$                                                  & 4                                                     & 4                                                         & 8                                            \\ 
$\{1,1,1,3\}$                                                    & 3                                                     & 1                                                         & 4                                            \\ 
$\{1,1,2,2\}$                                                    & 2                                                     & 2                                                         & 4                                            \\ 
$\{1,1,4\}$                                                      & 2                                                     & 0                                                         & 2                                            \\ 
$\{1,2,3\}$                                                      & 1                                                     & 1                                                         & 2                                            \\ 
$\{1,5\}$                                                        & 1                                                     & 0                                                         & 1                                            \\ 
$\{2,2,2\}$                                                      & 0                                                     & 0 or 4                                                         & 4                                            \\ 
$\{2,4\}$                                                        & 0                                                     & 0                                                         & 2                                            \\ 
$\{3,3\}$                                                        & 0                                                     & 1                                                         & 1                                            \\ 
$\{6\}$                                                          & 0                                                     & 0 or 1                                                        & 1                                            \\ 
\hline
\end{tabular}
\vspace{10pt}
\label{t1}
\caption{Number of tropes and singular points defined over the base field}
\end{table}

The number of tropes of each type follows directly from the combinatorics of the Weierstrass points, except in two subtle cases:
\begin{enumerate}
    \item The number of tropes of type $T_{ijk}$ can be $0$ or $4$ when $f(x)+\tfrac{1}{4}g(x)^2$ decomposes in $3$ different quadrics. The number of tropes is $4$ if and only if all quadrics split over the same quadratic number field, as in that case, assuming that the roots of the quadrics are $\{\omega_1,\omega_2\}$, $\{\omega_3,\omega_4\}$ and $\{\omega_5,\omega_6\}$, the tropes $\{T_{135}, T_{136}, T_{145}, T_{146}\}$ are defined over the field of definition of $\calC$.    
    \item The number of tropes of type $T_{ijk}$ can be $0$ or $1$ when $f(x)+\tfrac{1}{4}g(x)^2$ is irreducible. The number of tropes is $1$ if and only if the Galois group of the sextic is either $C_6$ or $S_3$, in which case there is a partition of the roots $\{\{\omega_i,\omega_j,\omega_k\},\{\omega_l,\omega_m,\omega_n\}\}$ preserved by the Galois group \cite{Awtrey2015IrreducibleResolvents}, and therefore $T_{ijk}$ is defined over the field of definition of $\calC$.
\end{enumerate}
The tropes for these special examples have been computed in \texttt{Examples.m}.\\

Consider the blow-up $Y\rightarrow X$. The preimage of every trope of $X$ is a line in $Y$ that we will denote by either $\hat{T}_i$ or $\hat{T}_{ijk}$.  Then, in $Y$ the tropes no longer intersect each other and they only intersect with the exceptional divisors $E_O$ and $E_{ij}$ according to the following rules:
\begin{align*}
E_O\cdot \hat{T}_i&=1, & E_O\cdot \hat{T}_{ijk}&=0, &
E_{ij}\cdot \hat{T}_i&=1, & E_{ij}\cdot \hat{T}_k&=0, &
E_{ij}\cdot \hat{T}_{ijk}&=1, & E_{ij}\cdot \hat{T}_{ikl}&=0.
\end{align*}

The translation $\tau_{ij}$ also acts on the exceptional divisors by the rules
\begin{align*}
\tau_{ij}(E_O)&=E_{ij}, &\tau_{ij}(E_{ij})&=E_{O}, &\tau_{ij}(E_{ik})&=E_{jk}, &\tau_{ij}(E_{kl})&=E_{mn},
\end{align*}
where $i,j,k,l,m,n$ are all distinct indices.
Let $H$ be the pull-back of a hyperplane section of $X$ under the blow-up map. Then, in $\Pic(Y)$, we can express the tropes in terms of $H$ and the $E_{ij}$ as
\begin{align*}
\hat{T}_i&=\tfrac{1}{2}(H-E_O-E_{ij}-E_{ik}-E_{il}-E_{im}-E_{in}),\\
\hat{T}_{ijk}&=\tfrac{1}{2}(H-E_{ij}-E_{ik}-E_{il}-E_{lm}-E_{ln}-E_{mn}).
\end{align*}

The Picard number of a Kummer surface is always $\rho+16$ where $\rho$ is the Picard number of the abelian surface of which it is the quotient. It is therefore possible to prove that, for a sufficiently general Kummer surface, $\Pic(Y)$ is generated over $\Z$ by the classes of the sixteen exceptional curves, the sixteen tropes, and the hyperplane section $H$ \cite{Keum1997AutomorphismsSurfaces}. \\

\section{Kummer surfaces over fields of characteristic two} \label{char2}

Let $\mathcal{C}$ be a genus two curve defined over a perfect field $k$ of characteristic two. The $2$-torsion of the Jacobian of $\calC$ satisfies that
\begin{align*}
\Jacc[2](\overline{k})\cong (\Z/2\Z)^r,
\end{align*}
where $0\leq r \leq2$ is what is known as the \textbf{$\boldsymbol{p}$-rank}. According to its $p$-rank, we say that $\calJ$ is \textbf{ordinary} if the $p$-rank is $2$, \textbf{almost ordinary} if it is $1$, and \textbf{supersingular} if it is $0$.\\

In each of the cases, the singular points of the quotient $X=\Jacc/\langle \iota\rangle$ have been found \cite{Katsura1978On2} to be the following:
\begin{itemize}
    \item In the ordinary case, $X$ has four rational singularities of type $D^1_4$ (in the sense of Artin \cite{Artin1975CoveringsP}).
    \item In the almost ordinary case, $X$ has two rational singularities of type $D^2_8$ (also in the sense of Artin).
    \item In the supersingular case, $X$ has one elliptic singularity of type $\Circled{4}^1_{\,0,1}$ in the sense of Wagreich \cite{Wagreich1970EllipticSurfaces} (in which case, the Kummer surface associated to $\Jacc$ is not a K3 surface).
\end{itemize}
For a general abelian surface $A$, not necessarily the Jacobian of a genus two curve, there is an additional possibility: $A$ may be supersingular and superspecial, i.e., isomorphic to the product of two supersingular elliptic curves, in which case $A/\langle \iota\rangle$ has an elliptic double singularity of type $\Circled{19}_{\,0}$. This situation cannot happen for Kummer surfaces associated to Jacobians of curves of genus two \cite[Theorem 3.3]{Ibukiyama1986SupersingularNumbers}.\\

In order to understand the resolution of singularities in these cases, Schröer observed that blowing up the schematic image of $\Jacc[2]$ inside $X$ generated a crepant partial resolution of the singularities \cite{Schroer2009TheSurfaces}. We claim that the equations for these partial resolutions can be obtained through a method analogous to the one that we described in Section \ref{sectionsp}.

\begin{theorem} \label{th_kummer_in_char_2}
Let $\calC$ be a genus-two curve over a perfect field of characteristic two, and let $X=\Jacc/\langle\iota\rangle$. Inside the subspace $\Ltwo(-2\Jacc[2])$ of sections vanishing on the $2$-torsion with multiplicity at least two, there exists a six-dimensional subspace whose associated linear system embeds a surface $Y\subset\mathbb{P}^5$ as the complete intersection of three quadrics. The resulting surface $Y$ is a partial desingularisation of $X$ and has the following singularities:
\begin{itemize}
\item If $\Jacc$ is ordinary, $Y$ has twelve singularities of type $A_1$. \item If $\Jacc$ is almost ordinary, $Y$ has two singularities of type $D_4^0$ and two singularities of type $A_3$. \item If $\Jacc$ is supersingular, $Y$ has an elliptic singularity of type $A_{\ast,o}+A_{\ast,o}+A_{\ast,o}+A_{\ast,o}+A_{\ast,o}$ in Laufer's notation \cite[Table 3]{Laufer1977OnSingularities}.
\end{itemize}
Furthermore, this embedding can be defined explicitly over the field of definition of the curve by considering the reduction modulo two of the characteristic zero model.
\end{theorem}
As mentioned in the introduction, Katsura and Kondō used the theory of line complexes to obtain similar results for Kummer surfaces that do not necessarily come from the Jacobians of genus two curves \cite{Katsura2023KummerTwo}. Furthermore, for Kummer surfaces coming from the Jacobians of ordinary genus two curves, they proved that  $\Ltwo(-2\Jacc[2])$ has exactly dimension six. The advantages of the method described in this article are that the scheme models that have been computed are defined over the field of definition of the curve, which is not always the case for the models of Katsura and Kondō, they have simpler equations, and also work for Jacobians of supersingular genus two curves.

\begin{proof}[Theorem \ref{th_kummer_in_char_2}]
The proof is constructive: given a genus two curve in characteristic two, we explicitly compute the equations of its Jacobian, the corresponding Kummer surface $X$, and its partial desingularisation $Y$. In section \ref{compsection}, we will describe an algorithm to compute a basis for $\Ltwo$ and in section \ref{PS}, we will describe the geometry and singularities of $Y$.
\end{proof}

\vspace{5pt}

\section{Computing a projective embedding of the Jacobian of a genus two curve} \label{compsection}

In this section, we will prove the following theorem:

\begin{theorem} \label{embeddjac}
Let $\calC$ be a genus-two curve over a perfect field $k$ of characteristic two. 
Then the linear system associated to the divisor $2(\Theta_+ + \Theta_-)$ defines an embedding over $k$ of the Jacobian variety $\Jacc$ as the intersection of $72$ quadrics in $\mathbb{P}^{15}$. Moreover, these quadrics can be written down explicitly.
\end{theorem}

\begin{proof}[Proof of Theorem \ref{embeddjac}]

Following the notation of Section \ref{sectionsp}, in characteristic $2$, given a Weierstrass point $w$, the divisor $\Theta_w$ is ample, and therefore, $4\Theta_w$ is very ample \cite[Theorem 1.1]{Pareschi2003RegularityEquations}. Since $\Theta_++\Theta_-\sim 2\Theta_w$ and $\ell(2(\Theta_++\Theta_-))=16$, the linear system $\lvert 2(\Theta_++\Theta_-)\rvert$ defines a projective embedding of $\Jacc$ into $\mathbb{P}^{15}$.\\

In order to compute this embedding we first compute a basis of $\Ltwo$ for a general genus two curve. Next, we determine all relations among the elements of this basis, thus obtaining the $72$ quadratic relations that these elements satisfy.\\

The two software packages that have been used to perform these computations have been Mathematica \cite{WolframResearch2024Mathematica14.0} for computing the majority of equations and Magma \cite{Bosma1997TheLanguage} to perform the more geometric operations on the Kummer surfaces such as checking the configuration of singularities described in Theorem \ref{th_kummer_in_char_2}.
All the relevant code is available \href{https://github.com/AlvaroGohe/Kummer-surfaces-and-Jacobians-of-genus-2-curves-in-characteristic-2}{here}.\\

\subsection{Computing a basis of \texorpdfstring{$\Ltwo$}{the space of sections}}

Let $\mathcal{C}$ be a general genus--two curve of the form
\[
y^{2} + \Bigl(\sum_{i=0}^{3} g_i x^{i}\Bigr) y \;=\; \sum_{i=0}^{6} f_i x^{i},
\]
which we assume satisfies $\deg(g)=3$; the relevance of this assumption will become apparent in the next section.\\

The idea of obtaining explicit models of Kummer surfaces in characteristic~$2$ by specialising from characteristic~$0$ goes back to Müller \cite{Muller2010ExplicitCharacteristic}.  
In his work, he computed the equations of a basis $\{k_1,k_2,k_3,k_4\}$ of $\Lone$ in characteristic zero such that, upon reducing the coefficients modulo two, form a basis $\{\overline{k}_1,\overline{k}_2,\overline{k}_3,\overline{k}_4\}$ of $\Lone$ for a general genus two curve defined over a field of characteristic two.\\

The basis $\{\overline{k}_1,\overline{k}_2,\overline{k}_3,\overline{k}_4\}$ is given by the following expressions:

\begin{align*}
\bar{k}_1&=1,& \bar{k}_2&=x_1+x_2,\\
\bar{k}_3&=x_1 x_2, & \bar{k}_4&=\frac{S(x_1,x_2)+y_2g(x_1)+y_1g(x_2)}{\left(x_1+x_2\right)^2},
\end{align*}
where
\begin{align*}
S(u,v)&=f_0+ f_1(u+v)+f_3uv(u+v)^2 +f_5u^2v^2(u+v)^2.\\
\end{align*}

Since these are linearly independent, and the product of any two elements of $\Lone$ lies in $\Ltwo$, we obtain ten elements of the basis of $\Ltwo$, which in analogy with the characteristic zero case, we will denote by $\overline{k}_{ij}$.\\

As $\ell(2(\Theta_++\Theta_-))=16$, we still need to compute six more independent elements of the basis, for which we will specialise from characteristic zero. While it is not known what the elements of these basis are for models of curves of the form $y^2+g(x)y=f(x)$, for genus two curves defined by equations of the form $y^2=\sum_{i=0}^6 \tilde{f}_i x^i$ we know equations for a basis of $\Ltwo$, and, more specifically, for a basis $\{b_1,\dots,b_6\}$ that generates all odd functions \cite[Section 3]{Flynn2012Two-coverings2}.\\

By considering the morphism $(x,y)\mapsto(x,y+\tfrac{1}{2}g(x))$, any curve $\mathcal{C}$ of the form $$y^2+g(x)y=f(x)$$ can be mapped over $k$ to a curve $\tilde{\mathcal{C}}$ of the form $y^2=\tilde{f}(x)$ where $\tilde{f}(x)=f(x)+\tfrac{1}{4}g(x)^2$. Through this change of coordinates, we can find a basis $\{b_1,\dots,b_6\}$ for the odd functions of $\Ltwo$ for models of curves of the form $y^2+g(x)y=f(x)$ in characteristic zero.\\

Although one might expect the reductions modulo two $\{\bar{b}_1,\dots,\bar{b}_6\}$ to remain independent, this fails in general. However, we can easily construct a basis that reduces well modulo two via the Algorithm \ref{alg:basis_normalisation}, which amounts to computing the Smith normal form associated to the basis.\\

Since the relations among the $b_i$ form a finitely generated $\mathbb{Z}$-module, the Smith normal form always exists, and the algorithm necessarily terminates after finitely many row and column operations. Thus, once the Smith normal form produces a primitive $\mathbb{Z}$-basis of this module, its reduction modulo two forms a basis of $\Ltwo$ over the residue field. In particular, this guarantees that the basis obtained by the algorithm specialises well to characteristic two.\newpage

\begin{algorithm}[H]
\caption{\textbf{- Computation of a basis that reduces well modulo two}}
\textbf{Input:} A basis $\mathcal{B} = \{b_1, \dots, b_n\}$ defined over $k$.\\
\textbf{Output:} A modified basis $\mathcal{B}'$ whose reduction modulo two is a basis of a subspace of $\Ltwo$ of the reduced curve $\tilde{\mathcal{C}}$.

\label{alg:basis_normalisation}
\begin{algorithmic}[1]
\State Multiply each $b_i \in \mathcal{B}$ by the smallest power of two that clears all denominators.
\State Reduce the coefficients of the resulting elements modulo two to obtain a new set $\{\bar{b}_1, \dots, \bar{b}_n\}$.
\State Compute all linear relations among the $\bar{b}_i$ by determining the kernel of the associated matrix over the reduced field.
\State Lift these relations to $k$ to produce new elements in $\mathcal{B}$ that vanish modulo two.
\State Divide by the appropriate powers of two to obtain new elements whose reductions modulo two are independent of the previous ones.
\State Repeat steps 1--5 until obtaining a basis of odd functions whose reductions are linearly independent and belong to $\Ltwo$.
\end{algorithmic}
\end{algorithm}

Applying this algorithm, we obtain a basis $\mathcal{B}$ such that the reductions of the elements of this basis modulo two are the following:
\begin{align*}
\bar{b}_1&=\frac{g(x_1)+g(x_2)}{x_1+x_2},\\
\bar{b}_2&=\frac{x_2g(x_1)+x_1g(x_2)}{x_1+x_2},\\
\bar{b}_3&=\frac{x_2^2g(x_1)+x_1^2g(x_2)}{x_1+x_2},\\
\bar{b}_4&=\frac{y_2g(x_1)+y_1g(x_2)}{x_1+x_2},\\
\bar{b}_5&=\frac{(g_1+g_2(x_1+x_2)+g_3x_1x_2)(y_2g(x_1)+y_1g(x_2))}{(x_1+x_2)^2}+\frac{T(x_1,x_2)}{x_1+x_2},\\
\bar{b}_6&=\frac{g_3(g_0 g_3+g_2^2(x_1+x_2)+g_2 g_3x_1x_2)(y_2g(x_1)+y_1g(x_2))}{(x_1+x_2)^2}+\frac{g_3^2R(x_1,x_2)}{x_1+x_2}.\\
\end{align*} 
where $T(u,v)$ and $R(u,v)$ are the following bivariate polynomials:\\
\begin{adjustbox}{width=\textwidth}
\begin{minipage}{\textwidth}
\begin{align*}
T(u,v)&= f_1g_1 + (f_3g_1+ f_1g_3) uv + (f_5g_1+ f_3g_3)u^2v^2 + f_5g_3u^3v^3 + (f_3g_0+ f_1g_2)(u + v)\\
& + g_0g_2g_3uv(u + v) + (f_5g_2+g_2^2g_3)u^2v^2(u + v) + f_1g_3(u + v)^2 + (f_5g_1+ g_1g_2g_3)uv(u + v)^2\\
&+ (f_5g_0+ g_0g_2g_3)(u + v)^3,\\
R(u,v)&=f_1 g_0 + (f_3 g_0 + f_1 g_2) u v + (f_5 g_0 + f_3 g_2) u^2 v^2 +  f_5 g_2 u^3 v^3 + (f_1 g_1 + g_0^2 g_2) (u + v)\\
&+ 
 (g_0 g_2^2 + g_0 g_1 g_3) u v (u + v) + (f_5 g_1  + 
 f_3 g_3 + g_1 g_2 g_3) u^2 v^2 (u + v) + f_1 g_2 (u + v)^2\\
 &+ 
 (f_5 g_0  + g_1^2 g_3  + 
 g_0 g_2 g_3) u v (u + v)^2 + (f_1 g_3+ g_0 g_1 g_3) (u + v)^3.\\
\end{align*}
\end{minipage}
\end{adjustbox}

However, working in characteristic two introduces an additional difficulty: the reductions $\overline{b}_i$ of the newly constructed elements $b_i$ turn out to be linearly dependent on the previously computed elements $\overline{k}_{jr}$. The underlying reason for this phenomenon is the following. The involution $\iota$ acts on $\Ltwo$ with eigenvalues $\pm 1$, and our chosen basis is diagonal with respect to this action. After reducing modulo two, the eigenvalues $1$ and $-1$ both become $1$, and therefore the action of $\iota$ on the reductions of these basis elements becomes trivial. This contradicts the known fact that $\Ltwo$ contains elements that are not fixed by~$\iota$.\\

This degeneration does not render the construction of the $\overline{b}_i$ useless. Indeed, the reductions $\overline{b}_i$ still play a key role in describing partial desingularisations of Kummer surfaces in characteristic two, as will be discussed in Section~\ref{PS}. Moreover, these elements allow us to recover additional generators of $\Ltwo$ in characteristic zero.\newpage Since each $\overline{b}_i$ can be written as a linear combination of the $\overline{k}_{jr}$, lifting these relations to characteristic zero produces elements of $\Ltwo$ whose reductions vanish in characteristic two and hence must be divisible by a power of two. Applying Algorithm~\ref{alg:basis_normalisation}, we obtain a basis $\{v_1,\dots,v_6\}$ whose reductions modulo two, together with the $\overline{k}_{jr}$, generate the desired basis of $\Ltwo$. The full procedure and the resulting equations are implemented in the Mathematica notebook \texttt{Part 1}.\\

\subsection{Computing the equations of the Jacobian}
We now need to compute the equations defining the embedding of the Jacobian in projective space. That is, we need to determine the $72$ quadratic relations that exist between the elements of $\Ltwo=\{\overline{v}_1,\dots,\overline{v}_6,\overline{k}_{11},\overline{k}_{12},\dots,\overline{k}_{44}\}$. Again, we will compute these from specialisation from the characteristic zero case, from the elements $\{v_1,\dots,v_6,k_{11},\dots,k_{44}\}$. The key to this is to first compute the relations in the basis that diagonalises the involution, $\{b_1,\dots,k_{44}\}$, as here working with odd and even functions greatly simplifies the process.\\

As described before, there are twenty Plücker-type relations of the form \begin{align*}
k_{ij}k_{rs}-k_{ir}k_{js}=0,
\end{align*}
which are straightforward. In order to compute the rest of relations, we adapted a strategy that Flynn \cite{Flynn1990TheField} originally used to compute these relations. Flynn observed that it was possible to define two independent weight functions on $x$, $y$ and the $f_i$ such that the equation of $\mathcal{C}$ has homogeneous weight.\\

As a consequence of this, all existing relations between the elements of a basis must also have homogeneous weights. Because there is only a limited number of monomials of a certain weight, this highly restricts the possible monomials involved in a relation.
We will extend this idea by defining two weight functions $w_1$ and $w_2$ on $x$, $y$, $f_i$ and $g_j$ by

\begin{table}[h]
\centering
\begin{tabular}{c|cccc|}
\cline{2-5}
                                                    & \cellcolor[HTML]{FFFFFF}$x$ & \cellcolor[HTML]{FFFFFF}$y$ & \cellcolor[HTML]{FFFFFF}$f_i$ & \cellcolor[HTML]{FFFFFF}$g_j$ \\ \hline
\multicolumn{1}{|c|}{\cellcolor[HTML]{FFFFFF}$w_1$} & $0$                         & $1$                         & $2$                           & $1$                           \\ 
\multicolumn{1}{|c|}{\cellcolor[HTML]{FFFFFF}$w_2$} & $1$                         & $3$                         & $6-i$                         & $3-j$                         \\ \hline
\end{tabular}
\vspace{10pt}
\caption{Values of the weight functions for $x$, $y$, $f_i$ and $g_j$}
\end{table}

From those weights, we can easily check that the weights of the elements of the basis of $\Ltwo$ are the following (note that the weight of the $k_{jr}$ are the sum of the weights of $k_j$ and $k_r$):

\begin{table}[h]
\centering
\begin{tabular}{c|cccccccccc|}
\cline{2-11}
                                                    & \cellcolor[HTML]{FFFFFF}$k_1$ & \cellcolor[HTML]{FFFFFF}$k_2$ & \cellcolor[HTML]{FFFFFF}$k_3$ & \cellcolor[HTML]{FFFFFF}$k_4$ & \cellcolor[HTML]{FFFFFF}$b_1$ & \cellcolor[HTML]{FFFFFF}$b_2$ & \cellcolor[HTML]{FFFFFF}$b_3$ & \cellcolor[HTML]{FFFFFF}$b_4$ & \cellcolor[HTML]{FFFFFF}$b_5$ & \cellcolor[HTML]{FFFFFF}$b_6$ \\ \hline
\multicolumn{1}{|c|}{\cellcolor[HTML]{FFFFFF}$w_1$} & $0$                           & $0$                           & $0$                           & $2$                           & $1$                           & $1$                           & $1$                           & $2$                           & $3$                           & $5$                           \\ 
\multicolumn{1}{|c|}{\cellcolor[HTML]{FFFFFF}$w_2$} & $0$                           & $1$                           & $2$                           & $4$                           & $2$                           & $3$                           & $4$                           & $5$                           & $6$                           & $7$                           \\ \hline
\end{tabular}
\vspace{10pt}
\caption{Values of the weight functions for $k_i$ and $b_j$}
\end{table}

We are looking for homogeneous relations between the elements of the basis. To avoid searching relations between rational functions, we will multiply all the $k_{jr}$ by $(x_1-x_2)^2$ and all the $b_{i}$ by $(x_1-x_2)^4$, so that all the functions are polynomials. From the relations described in the previous sections, we know that there are $21$ relations of the form 
$b_i b_j = \{\text{a quadratic polynomial in the } k_{sr}\}$. We start by computing the weights $w_1$ and $w_2$ corresponding to the product of $b_i b_j$ and we then compute all possible monomials on the variables $f_i$, $g_j$ and $k_{rs}$ of that weight. For example, 
\begin{align*}
w_1(b_1^2)&=2, & w_2(b_1^2)&=4,
\end{align*}
and the only monomials with those weights are
\begin{gather*}
\{g_1^2 k_{11}^2, g_0 g_2 k_{11}^2, f_2k_{11}^2,g_1g_2 k_{11}k_{12},g_0g_3 k_{11}k_{12},f_3 k_{11}k_{12},g_2^2k_{11}k_{13}, g_1 g_3 k_{11}k_{13},\\f_4k_{11}k_{13},k_{11}k_{14},
g_2^2k_{11}k_{22},g_1g_3k_{11}k_{22},f_4k_{11}k_{22},g_2 g_3k_{11}k_{23},f_5k_{11}k_{23},\\g_3^2k_{11}k_{33},f_6k_{11}k_{33},g_2 g_3k_{12}k_{22},f_5k_{12}k_{22},
g_3^2k_{12}k_{23},f_6 k_{12}k_{23}, g_3^2k_{22}^2,f_6k_{22}^2\}.\\
\end{gather*}
We therefore deduce that a $\Q$-linear combination of these elements must be equal to $b_1^2$. In order to compute this $\Q$-linear combination, we could expand the expressions of the $k_i$ in terms of $x_1, x_2,y_1$ and $y_2$ and find what this linear combination would have to be. This works for the products of $b_1,b_2$ and $b_3$ as their weights are small and there are not that many monomials with those weights. For instance, for $b_1^2$, we find that the relation that we are looking for is
\begin{align*}
b_1^2 &- 4 f_2 k_{11}^2 - g_1^2 k_{11}^2 - 4 f_3 k_{11}k_{12} - 2 g_1 g_2 k_{11} k_{12} - 
 4 f_4 k_{11} k_{22} - g_2^2 k_{11} k_{22}\\
 &- 2 g_1 g_3 k_{11} k_{22} - 4 f_5 k_{12}k_{22} - 
 2 g_2 g_3 k_{12} k_{22} - 4 f_6 k_{22}^2 - g_3^2 k_{22}^2 + 2 g_1 g_3 k_{11} k_{13}\\
 &+ 
 4 f_5 k_{11} k_{23} + 2 g_2 g_3 k_{11} k_{23} + 8 f_6 k_{12} k_{23} + 
 2 g_3^2 k_{12} k_{23} - 4 f_6 k_{11} k_{33} - g_3^2 k_{11} k_{33} - 4 k_{11} k_{14}=0.\\
\end{align*}
However, when we consider products involving $b_4$, $b_5$, or $b_6$, the direct approach becomes computationally infeasible, as the number of monomials of the corresponding weights grows rapidly. For example, there are $8374$ monomials of the same weight as $b_6^2$. Consequently, a more efficient method is required to determine the $\mathbb{Q}$-linear relations among the elements of the basis. We do this via the Algorithm \ref{alg:random_curves_points}.\\

The key idea of our approach is the following. Any $\mathbb{Q}$-linear relation among elements of $\Ltwo$ must hold after evaluating these elements on a randomly chosen curve together with two randomly chosen points $(x_1,y_1)$ and $(x_2,y_2)$ lying on it. Thus, if we evaluate all the relevant monomials of a certain weight, the resulting values must also satisfy the desired linear relation. A single evaluation yields a vector in $\mathbb{Q}^n$, where $n$ is the number of monomials of weight $(w_1,w_2)$, and therefore satisfies many accidental relations. However, by repeating this process over many independently chosen curves and points, we obtain a collection of vectors whose only common linear dependencies are precisely the genuine $\mathbb{Q}$-linear relations we seek. Once we have generated more than $n$ such vectors, the kernel of the corresponding matrix consists exactly of these relations.\\

\begin{algorithm}[H]
\caption{\textbf{- Generation of relations between the elements of the basis}}
\textbf{Input:} The elements of a basis $\mathcal{B}=\{k_{11},\dots,k_{44},b_1,\dots,b_6\}$ of $\Ltwo$ and an integer $M\in\mathbb{N}$ (in our computations we used $M=4$). \\
\textbf{Output:} All relations of the form $b_i b_j = \{\text{quadratic polynomial in the } k_{rs}\}$.

\label{alg:random_curves_points}
\begin{algorithmic}[1]
\ForAll{$1\leq i\leq j\leq 6$}
\State Compute the weight $(w_1,w_2)$ of $b_i b_j$.
\State Compute all the possible monomials on the variables of weight $(w_1,w_2)$, and store them in a vector $\calV$.
\State Set $g_1,\dots,g_3,f_1,\dots,f_6$ to be random integers in the interval $[-M,M]$.
\State Pick random integers $x_1,y_1,x_2$ in $[-M,M]$ and set $y_2 \gets y_1 \pm 1$.
\State Set $f_0$ and $g_0$ to be
\begin{align*}
f_0&\gets\frac{(y_2^2+(\sum_{j=1}^3 g_j x_2^j)y_2-\sum_{i=1}^6 f_i x_2^i)y_1-(y_1^2+(\sum_{j=1}^3 g_j x_1^j)y_1-\sum_{i=1}^6 f_i x_1^i)y_2}{y_1-y_2}.\\
g_0&\gets\frac{(y_2^2+(\sum_{j=1}^3 g_j x_2^j)-\sum_{i=1}^6 f_i x_2^i)-(y_1^2+(\sum_{j=1}^3 g_j x_1^j)y_1-\sum_{i=1}^6 f_i x_1^i)}{y_1-y_2}.
\end{align*}
\State  Evaluate all monomials in $\mathcal{V}$ at the two points $(x_1,y_1)$ and $(x_2,y_2)$ on the curve
$y^2 + \Big(\sum_{i=0}^3 g_i x^i\Big)y = \sum_{i=0}^6 f_i x^i.$
\State Repeat steps 3--7 until more than $n=\lvert\mathcal{V}\rvert$ evaluation vectors have been generated.
\State Form the matrix whose rows are these vectors and compute its kernel.
\State Select among the elements of this kernel the one that gives you a relation of the form $b_i b_j = \{\text{quadratic polynomial in the } k_{rs}\}$.
\EndFor
\end{algorithmic}
\end{algorithm}

Using this algorithm to compute all of the relations between elements of weights $w_1=4$ and degree $w_2=10$, we can also recover the relation defining the Kummer surface $X\subset\mathbb{P}^3$.\\

Up to this point, we have determined $42$ of the $72$ quadratic relations cutting out the Jacobian in $\mathbb{P}^{15}$. The remaining $30$ relations are precisely those involving only monomials of the form $k_{ij} b_s$ with $1 \leq i,j \leq 4$ and $1 \leq s \leq 6$. To obtain these, we first compute the eight relations among the elements $k_i b_s$. By multiplying each of these relations by $k_1, k_2, k_3$ and $k_4$, we obtain a collection of $32$ new relations; removing the two linear dependencies among them leaves exactly the desired $30$ independent relations.\\

Altogether, this yields a complete set of $72$ quadratic equations defining a projective model of the Jacobian valid over any field of characteristic different from two. To obtain the model in characteristic two, the only remaining step is to rewrite these relations in terms of the basis $\{v_1,\dots,v_6\}$ produced by Algorithm~\ref{alg:basis_normalisation} in place of the original basis $\{b_1,\dots,b_6\}$, and then to take suitable linear combinations so that their reductions modulo two give the correct defining equations of the Jacobian. As a consistency check, we verified in Magma that, for several random curves in characteristic~$2$, these $72$ quadrics cut out a smooth surface birational to $\Jacc$.\\

The computations to find the equations are contained in the Mathematica notebook \texttt{Part 2}, and the equations can be found in the file \texttt{72 equations of the Jacobian.txt}. This embedding has also been implemented in the Magma file \texttt{Functions.m}, together with several auxiliary routines linking this projective model with the existing Magma machinery for Jacobians.
\end{proof}
\vspace{5pt}

\section{Partial desingularisations of Kummer surfaces in characteristic two }\label{PS}

In this section, we provide a detailed proof of Theorem \ref{th_kummer_in_char_2} by constructing the partial desingularisation $Y$ and studying its geometry.\\

\subsection{Computing models of Kummer surfaces in characteristic two} \label{secblowup}
The functions $k_i$ introduced in Section \ref{compsection} generate an embedding of the Kummer surface $X$ into $\mathbb{P}^3$ as the zero locus of a quartic polynomial. Upon reduction modulo two, this quartic coincides with the model obtained by Duquesne \cite{Duquesne2010Traces2}, and exhibits precisely the singularities classified by Katsura \cite{Katsura1978On2}: four singularities of type $D_4^1$ in the ordinary case, two of type $D_8^2$ in the almost ordinary case, and one singularity of type $\Circled{4}^1_{\,0,1}$ in the supersingular case. As described in Section~\ref{sectionsp}, we have a rational map
\begin{align*}
\Jacc&\xdashrightarrow{\hspace{7mm}}\mathbb{P}^5\\
D&\longmapsto[b_1(D):\dots:b_6(D)]
\end{align*}
This map induces a closed embedding of a surface $Y$ inside $\mathbb{P}^5$ as the complete intersection of three quadrics and there is a degree four birational map from $X$ to $Y$ which is defined outside of the singular locus of $X$. The equations of the scheme $Y$ can be accessed in Magma via the function \texttt{DesingularisedKummer}.\\

Now, consider the reduction of the functions $b_i$ in the reduced curve $\overline{\calC}$, which we will denote by $\overline{b}_i$. While these are still linearly independent by construction, the first major difference with respect to the characteristic zero case is that, while in characteristic zero the $b_i$ did not belong to $\Sym^2\Lone$, the space of quadratic functions in $\{k_1,k_2,k_3,k_4\}$, all the $\overline{b}_i$ can be expressed as quadratic functions in the $\{\overline{k}_1,\overline{k}_2,\overline{k}_3,\overline{k}_4\}$. The rational map
\begin{align*}
\Jacc&\xdashrightarrow{\hspace{7mm}}\mathbb{P}^5\\
D&\longmapsto[\overline{b}_1(D):\dots:\overline{b}_6(D)]
\end{align*}
defines an embedding of a surface $Y$ inside $\mathbb{P}^5$ as the complete intersection of three quadrics, but unlike in the characteristic zero case, this surface $Y$ is not smooth. However, this map is still of interest, as all the $\overline{b}_i$ are simultaneously zero precisely at the points corresponding to $\mathcal{J}[2]$, and therefore the indeterminacy locus of the map
\begin{align*}
X&\xdashrightarrow{\hspace{7mm}}Y\\
[\overline{k}_1:\dots:\overline{k}_4]&\longmapsto[\overline{b}_1:\dots:\overline{b}_6]
\end{align*}
coincides with the singular locus of $X$. The inverse of this map, which we will denote by $\varphi$ is a blow-up of the singular locus, which we will later describe.\\

We have computed explicit equations defining this map, from the fact that the function $(2y_1+g(x_1))(2y_2+g(x_2))\overline{k}_i$ can be expressed as a polynomial in $\{\overline{b}_1,\overline{b}_2,\overline{b}_3,\overline{b}_4\}$. Since this map involves only the first four $\overline{b}_i$, the projection map from $\mathbb{P}^5$ to $\mathbb{P}^3$ consisting of taking the first four coordinates descends into a rational map
\begin{align*}
Y&\xdashrightarrow{\hspace{7mm}}W\subset\mathbb{P}^3\\
[\overline{b}_1:\dots:\overline{b}_6]&\longmapsto[\overline{b}_1:\dots:\overline{b}_4]
\end{align*}
where $W$ is a quartic surface in $\mathbb{P}^3$ which, by similarity with the characteristic zero case, we will call the Weddle surface. We will analyse its features according to the $p$-rank of the curve in Section \ref{sectionweddle}.\\

In order to describe what the partial desingularisations look like, it will be convenient to analyse separately the cases according to the $p$-rank.  The following proposition will be useful:
\begin{proposition} \label{proproots}
Let $\calC$ be a genus two curve of the form $y^2+g(x)y=f(x)$ with $\deg(g)=3$. Then, $\Jacc$ is ordinary, almost ordinary or supersingular, according to whether $g(x)$ has three, two or one distinct roots.
\end{proposition}
\begin{proof}[Proof of Proposition \ref{proproots}]
As in the characteristic zero case, it is easy to see that any non-zero $2$-torsion point is of the form $D_{ij}=(w_i)+(w_j)-K_\calC$ where $\{w_i,w_j\}$ is an unordered pair of Weierstrass points of $\calC$. Since $\iota(x,y) = (x,y+g(x))$ fixes a point if and only if $g(x)=0$, the Weierstrass points are precisely the roots of $g(x)$. Hence, over the splitting field of $g$, there are $\binom{3}{2}=3$ non-trivial $2$-torsion points if $g$ has three distinct roots, $\binom{2}{2}=1$ non-trivial $2$-torsion points if $g$ has two distinct roots and no non-trivial $2$-torsion points if $g$ only has one root. 
\end{proof}

For models with $\deg(g)<3$, analogous statements hold. Given any genus two curve in characteristic two with affine model $y^2 + g(x)y = f(x)$, one can always apply a change of variables that moves the point at infinity to a finite point and preserves the Weierstrass points, hence obtaining a model with $\deg(g)=3$.\\

We now turn to the geometry of the associated Kummer surfaces. All computations appearing in this section were carried out in the Mathematica notebook \texttt{Part 3}, and the resulting equations are available in Magma through the function \texttt{Lines}. These equations may be used to verify the results presented below, as illustrated in \texttt{Examples.m}.\\

\subsection{The geometry of Kummer surfaces in the ordinary case} \label{ssord}
Let $\mathcal{C}$ be an ordinary genus two curve of the form
\begin{align*}
y^2+(\sum_{j=0}^3g_jx^j)y=\sum_{j=0}^6f_jx^j,
\end{align*}
so that the Weierstrass points have coordinates $(\alpha_i,\beta_i)$, where $1\leq i\leq3$ and $\beta_i=\sqrt{\sum_{j=0}^6f_j\alpha^j_i}$. Note that, by Proposition \ref{proproots}, these $\alpha_i$ correspond to the three distinct roots of $g$. As in the characteristic zero case, the $2$-torsion points of $\Jacc$ are of the form $D_{ij}=(w_i)+(w_j)-K_\mathcal{C}$ where $\{w_i,w_j\}$ are Weierstrass points whose coordinates are $(\alpha_i,\beta_i)$ and $(\alpha_j,\beta_j)$, and each of these corresponds to a singular point $P_{ij}$ of the Kummer surface $X$ associated to $\mathcal{C}$.\\

In our model, these singular points are defined over $k(\alpha_i+\alpha_j, \alpha_i\alpha_j)$, and their coordinates are given by
\begin{align*}
P_O &= [0:0:0:1], & 
P_{ij} &= \Big[1:\alpha_i+\alpha_j:\alpha_i \alpha_j:\frac{f_1+\alpha_i \alpha_j f_3 + \alpha_i^2 \alpha_j^2 f_5}{\alpha_i+\alpha_j}\Big].
\end{align*}

In characteristic two, it still makes sense to talk about tropes in $X$: we can define $T_i$ to be the image of $\calC\times\{w_i\}$ under the composition of the maps $\calC^{(2)}\rightarrow\Jacc$ and $\Jacc\rightarrow X$. Then, $T_i$ is a conic in $X$ which goes through the points $P_O$, $P_{ij}$ and $P_{ik}$ where the indices $\{i,j,k\}$ are all distinct. This trope could also be defined in an alternative way by considering the unique plane going through $P_O$, $P_{ij}$ and $P_{ik}$ (whenever the roots of $g$ are distinct, these points are not collinear), which intersects $X$ in the conic $T_i$ with multiplicity two. Similarly, one defines the fourth trope $T_{123}$ as the conic through the singular points $P_{12}$, $P_{13}$, and $P_{23}$.\\

As in characteristic zero, translation by a $2$-torsion point $D_{ij}$ induces a linear automorphism $\tau_{ij}$ on the Kummer surface, permuting the tropes according to
\begin{align*}
\tau_{ij}(T_i)&=T_j, & \tau_{ij}(T_{123})&=T_k,
\end{align*}
where $\{i,j,k\}$ are all distinct indices. In our model, the tropes are defined by the intersection of $X$ with the following planes:\\
\begin{adjustbox}{width=\textwidth}
\begin{minipage}{1.1\textwidth}
\begin{align*}
\pi_1&=\alpha_1^2 \overline{k}_1 + \alpha_1 \overline{k}_2 + \overline{k}_3=0,\\
\pi_2&=\alpha_2^2 \overline{k}_1 + \alpha_2 \overline{k}_2 + \overline{k}_3=0,\\
\pi_3&=\alpha_3^2 \overline{k}_1 + \alpha_3 \overline{k}_2 + \overline{k}_3=0,\\
\pi_{123}&=(f_1 + f_3 + f_5) g_2^2 \overline{k}_1 + 
 g_2 (f_5 g_1 + f_1 g_3 + f_3 g_3) \overline{k}_2 + (f_5 g_1^2 + f_3 g_1 g_3 + f_1 g_3^2) \overline{k}_3 + 
 g_2 (g_1 + g_3) \overline{k}_4=0.\\
 \end{align*}
 \end{minipage}
 \end{adjustbox}
Depending on the decomposition of $g(x)$ into irreducible factors, the number of tropes and singular points defined over the ground field is as follows:

\begin{table}[h]
\centering
\begin{tabular}{|c|ccc|}
\hline
\rowcolor[HTML]{FFFFFF} 
Partition                             & $\#$ tropes of type $T_i$ & $\#$ tropes of type $T_{ijk}$ & $\#$ singular points \\ \hline
\cellcolor[HTML]{FFFFFF}$\{1,1,1\}$   & 3                         & 1                             & 4                    \\ 
\cellcolor[HTML]{FFFFFF}$\{1,2\}$     & 1                         & 1                             & 2                    \\ 
\cellcolor[HTML]{FFFFFF}$\{3\}$ & 0                         & 1                             & 1                    \\ \hline
\end{tabular}
\vspace{10pt}
\caption{Number of tropes and singular points defined over the base field in characteristic two}
\end{table}
An alternative description of these tropes arises via specialisation from the characteristic zero case.
Consider a curve $\calC$ defined over a discrete valuation ring whose fraction field $K$ is complete and has a perfect residue field of characteristic two, such that all the $2$-torsion is defined over $K$ and such that $\calC$ has good ordinary reduction.\\

The Weierstrass points of $\calC$ are the roots of $4f(x)+g(x)^2$, which reduce $2$-to-$1$ to the Weierstrass points of the reduced curve $\overline{\calC}$ whose $x$-coordinates are roots of $g(x)$. On the Jacobian, this gives a $4$-to-$1$ reduction of the $2$-torsion points. Correspondingly, the singular points of $X$ reduce $4$-to-$1$, and each trope reduces to the conic passing through the reductions of its singular points.\\

Now consider the blow-up $\varphi$ that was described in the previous subsection. 
The exceptional divisors associated to the resolution of a $D_4^1$ singularity form a tree configuration.  For each of the four $D_4^1$ singularities, the partial desingularisation map blows up the central exceptional curve of each of the singularities; therefore, the partial desingularisation has twelve $A_1$ singularities.\\

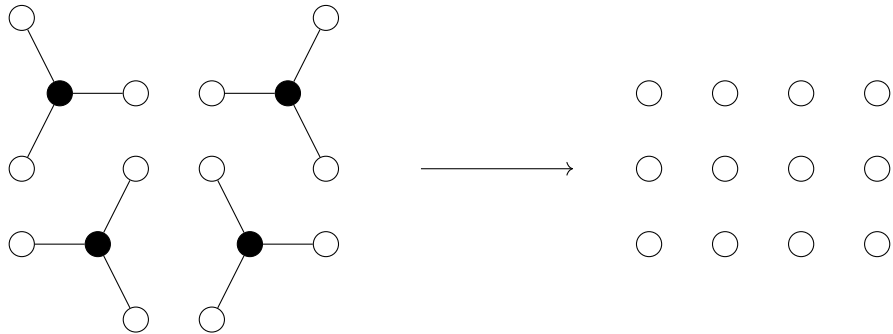
\begin{center}
\begin{tikzpicture}
	\begin{pgfonlayer}{nodelayer}
		\node [style=none] (48) at (4.75, -1) {};
		\node [style=black simple style] (62) at (0, 0) {};
		\node [style=white with black border] (63) at (1, 0) {};
		\node [style=white with black border] (64) at (-0.5, 1) {};
		\node [style=white with black border] (65) at (-0.5, -1) {};
		\node [style=white with black border] (66) at (-0.5, -2) {};
		\node [style=black simple style] (67) at (0.5, -2) {};
		\node [style=white with black border] (68) at (1, -3) {};
		\node [style=white with black border] (69) at (1, -1) {};
		\node [style=white with black border] (71) at (2, 0) {};
		\node [style=black simple style] (72) at (3, 0) {};
		\node [style=white with black border] (73) at (3.5, -1) {};
		\node [style=white with black border] (74) at (3.5, 1) {};
		\node [style=black simple style] (75) at (2.5, -2) {};
		\node [style=white with black border] (76) at (3.5, -2) {};
		\node [style=white with black border] (77) at (2, -1) {};
		\node [style=white with black border] (78) at (2, -3) {};
		\node [style=none] (80) at (6.75, -1) {};
		\node [style=white with black border] (81) at (7.75, -1) {};
		\node [style=white with black border] (82) at (7.75, 0) {};
		\node [style=white with black border] (83) at (7.75, -2) {};
		\node [style=white with black border] (84) at (8.75, 0) {};
		\node [style=white with black border] (85) at (8.75, -1) {};
		\node [style=white with black border] (86) at (8.75, -2) {};
		\node [style=white with black border] (87) at (9.75, -1) {};
		\node [style=white with black border] (88) at (9.75, 0) {};
		\node [style=white with black border] (89) at (9.75, -2) {};
		\node [style=white with black border] (90) at (10.75, 0) {};
		\node [style=white with black border] (91) at (10.75, -1) {};
		\node [style=white with black border] (92) at (10.75, -2) {};
	\end{pgfonlayer}
	\begin{pgfonlayer}{edgelayer}
		\draw (62) to (64);
		\draw (62) to (65);
		\draw (62) to (63);
		\draw (66) to (67);
		\draw (67) to (68);
		\draw (67) to (69);
		\draw (71) to (72);
		\draw (72) to (73);
		\draw (72) to (74);
		\draw (75) to (77);
		\draw (75) to (78);
		\draw (75) to (76);
		\draw [style=arrow] (48.center) to (80.center);
	\end{pgfonlayer}
\end{tikzpicture}
\end{center}
\vspace{10pt}

From the explicit equations that we have computed, it is easy to check that the image of each of the conics corresponding to the tropes of $X$ is a line of $Y$.\newpage Then, these twelve singularities are nodes that lie in the intersection points of the four exceptional divisors associated with the singularities of the Kummer surface, and the image of the four tropes. All tropes and exceptional lines of $Y$ lie in the hyperplane section\\
\begin{adjustbox}{width=\textwidth}
\begin{minipage}{1.1\textwidth}
\begin{gather*}
(\alpha_2^2 \alpha_3 \beta_1 + \alpha_2 \alpha_3^2 \beta_1 + \alpha_1^2 \alpha_3 \beta_2 + \alpha_1 \alpha_3^2 \beta_2 + \alpha_1^2 \alpha_2 \beta_3 + \alpha_1 \alpha_2^2 \beta_3) \overline{b}_1 +(\alpha_2^2 \beta_1 + \alpha_3^2 \beta_1 + \alpha_1^2 \beta_2 + \alpha_3^2 \beta_2 + \alpha_1^2 \beta_3 + \alpha_2^2 \beta_3)\overline{b}_2\\
+(\alpha_2 \beta_1 + \alpha_3 \beta_1 + \alpha_1 \beta_2 + \alpha_3 \beta_2 + \alpha_1 \beta_3 + \alpha_2 \beta_3)\overline{b}_3+(\alpha_1 + \alpha_2) (\alpha_1 + \alpha_3) (\alpha_2 + \alpha_3) \overline{b}_4  =0.\\
\end{gather*}
\end{minipage}
\end{adjustbox}

As described by Katsura and Kondō, if we denote by $\{E_O, E_{12}, E_{13}, E_{23}\}$ the exceptional divisors corresponding to the singular points of $X$ and by $\{\hat{T}_1,\hat{T}_2,\hat{T}_3,\hat{T}_{123}\}$ the images of the tropes in $Y$, then these divisors intersect according to the following configuration:\\

\begin{center}
\begin{tikzpicture}
	\begin{pgfonlayer}{nodelayer}
		\node [style=none] (16) at (-2.5, 2) {};
		\node [style=none] (18) at (-1, 2.5) {};
		\node [style=none] (19) at (0, 2.5) {};
		\node [style=none] (20) at (1, 2.5) {};
		\node [style=none] (21) at (1.5, 2) {};
		\node [style=none] (22) at (1.5, 1) {};
		\node [style=none] (23) at (1.5, 0) {};
		\node [style=none] (24) at (1, -1.5) {};
		\node [style=none] (25) at (0, -1.5) {};
		\node [style=none] (26) at (-1, -1.5) {};
		\node [style=none] (27) at (-2.5, 1) {};
		\node [style=none] (28) at (-2.5, 0) {};
		\node [style=none] (30) at (-2.5, -1) {};
		\node [style=none] (32) at (1.5, -1) {};
		\node [style=none] (34) at (-2, -1.5) {};
		\node [style=none] (39) at (-2, 2.5) {};
		\node [style=none] (40) at (-2, 2.75) {$E_O$};
		\node [style=none] (42) at (0, 2.75) {$E_{13}$};
		\node [style=none] (43) at (1, 2.75) {$E_{23}$};
		\node [style=none] (44) at (-3, 2) {$\hat{T}_{1}$};
		\node [style=none] (45) at (-3, 1) {$\hat{T}_{2}$};
		\node [style=none] (46) at (-3, 0) {$\hat{T}_{3}$};
		\node [style=none] (47) at (-3, -1) {$\hat{T}_{123}$};
		\node [style=none] (48) at (-1, 2.75) {$E_{12}$};
		\node [style=none] (50) at (-2.25, -1) {};
		\node [style=none] (51) at (-1.75, -1) {};
		\node [style=none] (52) at (-1.25, 0) {};
		\node [style=none] (53) at (-0.75, 0) {};
		\node [style=none] (54) at (-0.25, 1) {};
		\node [style=none] (55) at (0.25, 1) {};
		\node [style=none] (56) at (0.75, 2) {};
		\node [style=none] (57) at (1.25, 2) {};
	\end{pgfonlayer}
	\begin{pgfonlayer}{edgelayer}
		\draw (39.center) to (34.center);
		\draw (18.center) to (26.center);
		\draw (19.center) to (25.center);
		\draw (20.center) to (24.center);
		\draw (21.center) to (57.center);
		\draw (56.center) to (16.center);
		\draw (22.center) to (55.center);
		\draw (54.center) to (27.center);
		\draw (23.center) to (53.center);
		\draw (52.center) to (28.center);
		\draw (32.center) to (51.center);
		\draw (50.center) to (30.center);
	\end{pgfonlayer}
\end{tikzpicture}\vspace{5pt}
\end{center}

From this, we can deduce that the minimal resolution of the Kummer surface contains twenty $(-2)$-curves which are the proper transforms of the eight lines described above and the twelve exceptional curves that we obtain from blowing up the singular points.
These curves meet according to the following dual graph:\\

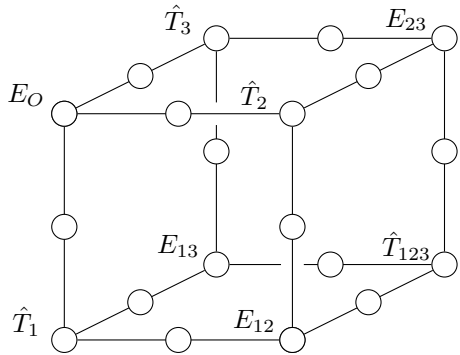
\begin{center}
\begin{tikzpicture}
	\begin{pgfonlayer}{nodelayer}
		\node [style=white with black border] (3) at (1, 2) {};
		\node [style=white with black border] (9) at (4, -1) {};
		\node [style=none] (24) at (0.5, 2.25) {$E_O$};
		\node [style=none] (26) at (2.5, 3.25) {$\hat{T}_3$};
		\node [style=none] (28) at (0.5, -0.75) {$\hat{T}_1$};
		\node [style=none] (30) at (3.5, -0.75) {$E_{12}$};
		\node [style=none] (31) at (2.5, 0.25) {$E_{13}$};
		\node [style=none] (32) at (3.5, 2.25) {$\hat{T}_2$};
		\node [style=none] (33) at (5.5, 3.25) {$E_{23}$};
		\node [style=none] (34) at (5.5, 0.25) {$\hat{T}_{123}$};
		\node [style=white with black border] (35) at (4, 2) {};
		\node [style=white with black border] (36) at (4, -1) {};
		\node [style=white with black border] (37) at (5, 2.5) {};
		\node [style=white with black border] (38) at (6, 3) {};
		\node [style=white with black border] (39) at (4.5, 3) {};
		\node [style=white with black border] (40) at (3, 3) {};
		\node [style=white with black border] (41) at (2, 2.5) {};
		\node [style=white with black border] (42) at (2.5, 2) {};
		\node [style=white with black border] (43) at (1, 2) {};
		\node [style=white with black border] (44) at (4, 0.5) {};
		\node [style=white with black border] (45) at (1, 0.5) {};
		\node [style=white with black border] (46) at (1, -1) {};
		\node [style=white with black border] (47) at (2, -0.5) {};
		\node [style=white with black border] (48) at (3, 0) {};
		\node [style=white with black border] (49) at (2.5, -1) {};
		\node [style=white with black border] (55) at (3, 1.5) {};
		\node [style=white with black border] (56) at (6, 0) {};
		\node [style=white with black border] (57) at (6, 1.5) {};
		\node [style=white with black border] (58) at (4.5, 0) {};
		\node [style=white with black border] (60) at (5, -0.5) {};
		\node [style=none] (61) at (3, 1.85) {};
		\node [style=none] (62) at (3, 2.15) {};
		\node [style=none] (63) at (3.85, 0) {};
		\node [style=none] (64) at (4.15, 0) {};
	\end{pgfonlayer}
	\begin{pgfonlayer}{edgelayer}
		\draw (43) to (40);
		\draw (40) to (38);
		\draw (43) to (35);
		\draw (35) to (38);
		\draw (43) to (46);
		\draw (46) to (36);
		\draw (36) to (35);
		\draw (46) to (48);
		\draw (36) to (56);
		\draw (56) to (38);
		\draw (48) to (61.center);
		\draw (62.center) to (40);
		\draw (48) to (63.center);
		\draw (64.center) to (56);
	\end{pgfonlayer}
\end{tikzpicture}
\end{center} \vspace{5pt}

Katsura and Kondō further showed that a general Kummer surface in $\mathbb{P}^3_{x,y,z,t}$ can be written as
\begin{gather} \label{KKord}
(a_1+a_2)^2(c_3 x^2y^2+d_3 z^2 t^2)+(a_1+a_3)^2(c_2 x^2z^2+ d_2 y^2t^2)\\
+(a_2+a_3)^2(c_1 x^2t^2+ d_1 y^2z^2)+(a_1+a_2)(a_2+a_3)(a_3+a_1)xyzt=0.\nonumber\\ \nonumber
\end{gather}

In this model, the planes defining the tropes of the Kummer are given by the equations $x=0$, $y=0$, $z=0$ and $t=0$. We found a linear projective map $\psi$ which is an isomorphism between $X$ and the variety defined in equation \eqref{KKord} for some choice of parameters. The Cremona transformation $\phi$ described by Katsura and Kondō
\begin{align*}
\phi([x:y:z:t])=\Big[\sqrt{d_1d_2d_3}\;yzt:\sqrt{c_1c_2d_3}\;xzt:\sqrt{c_1d_2c_3}\;xyt:\sqrt{d_1c_2c_3}\;xyz\Big]
\end{align*} 
induces a corresponding Cremona transformation on $X$ via the conjugation $\psi^{-1} \circ \phi \circ \psi$. Similarly, the linear actions $\tau_{ij}$ induced by addition of a $2$-torsion point on $\Jacc(\mathcal{C})$ correspond to linear automorphisms on $Y$.  
\\

Finally, Katsura and Kondō described the partial desingularisation of \eqref{KKord} as a complete intersection
\begin{align*}
\sum_{i=1}^3 X_i Y_i=\sum_{i=1}^3 a_i X_i Y_i+c_i X_i^2+d_i Y_i^2=\sum_{i=1}^3 a_i^2 X_i Y_i=0.\\
\end{align*}
which can also be connected to our model $Y$ through an explicit change of variables. Katsura and Kondō described three automorphisms $\iota_1$, $\iota_2$, $\iota_3$ generating $(\Z/2\Z)^3$ which correspond, under this change, to the linear actions on $Y$ induced by translations by the $2$-torsion points, and the Cremona transformation that interchanges tropes with exceptional divisors.\\

\subsection{The geometry of Kummer surfaces in the almost ordinary case} \label{almost ord section}
In the almost ordinary case, Proposition~\ref{proproots} shows that the cubic polynomial $g(x)$ has two distinct roots over its splitting field: a simple root $\alpha_{1}$ and a double root $\alpha_{2}$. The asymmetry between these two roots is not immediately reflected in the geometry of the associated surfaces, but it will soon become apparent.
\\
 
 The only non-trivial $2$-torsion point of $\Jacc$ is $D_{12}=(w_1)+(w_2)-K_\mathcal{C}$ where $w_{1} = (\alpha_{1},\beta_{1})$ and $w_{2} = (\alpha_{2},\beta_{2})$ are Weierstrass points. This point corresponds to a singularity $P_{12}$ of the Kummer surface $X$ attached to $\mathcal{C}$. Together with the point $P_{O}$ corresponding to the identity of the group law, these are the two $D_{8}^{2}$ singularities of $X$.
Assuming that we work over a perfect ground field $k$, both singular points are defined over $k$. Indeed, writing
 \begin{align*}
 g(x)=g_3(x-\alpha_1)(x-\alpha_2)^2=g_3x^3+ g_3\alpha_1x^2 + g_3 \alpha_2^2x +g_3\alpha_1\alpha_2^2.
 \end{align*}
we obtain
\begin{align*}
\alpha_1&=\frac{g_2}{g_3}, & \alpha_2&=\sqrt{\frac{g_1}{g_3}}.
\end{align*}

As in the ordinary case, we define $T_{i}$ to be the image of $\mathcal{C}\times\{w_{i}\}$ under the composition $\mathcal{C}^{(2)} \to \Jacc \to X$. The curves $T_{1}$ and $T_{2}$ are conics on $X$, defined over $k$, and both pass through $P_O$ and $P_{12}$. A convenient way to compute their equations is to specialise from the ordinary case. Consider an ordinary curve given by
\begin{align*}
y^2+g_3(x-\alpha_1)(x-\alpha_2)(x-\alpha_3)=f(x).
\end{align*}
When we specialise this curve by letting $\alpha_{3}\mapsto\alpha_{2}$, the pairs of points $(P_{O}, P_{23})$ and $(P_{12}, P_{13})$ degenerate to the same limiting points $P_{O}$ and $P_{12}$, respectively. Likewise, the families of planes $T_{1}$ and $T_{123}$ degenerate to the same trope $T_{1}$, and $T_{2}$ and $T_{3}$ degenerate to $T_{2}$.
\\

This description shows that $T_{1}$ and $T_{2}$ meet $P_O$ and $P_{12}$ with different multiplicities. For instance, $T_{1}$ and $T_{123}$ both pass through $P_{12}$ and $P_{13}$, which specialise to $P_{12}$, and through a third point specialising to $P_{O}$. Hence $T_{1}$ meets $P_{12}$ with higher multiplicity than $P_O$. A symmetric argument shows that $T_{2}$ meets $P_O$ with higher multiplicity than $P_{12}$. These differing multiplicities will influence the singularities appearing after blowing up~$X$.\\

We therefore obtain a natural $2$-to-$1$ specialisation from the ordinary case, or, viewed as a reduction from characteristic zero, an $8$-to-$1$ specialisation of both tropes and singular points.\\

Consider now the blow-up described in~\ref{secblowup}. The exceptional divisors resolving a $D_{8}^{2}$ singularity form a tree, and the partial desingularisation blows up one of the central exceptional curves at each $D_{8}^{2}$ point.\newpage

As a consequence, each $D_8^2$ gets blown-up into a $D_4^0$ and an $A_3$ singularity.\\

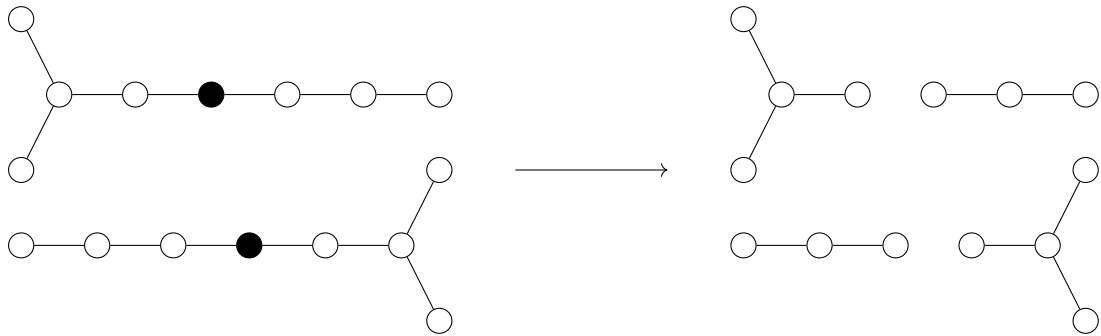
\begin{center}
\begin{tikzpicture}
	\begin{pgfonlayer}{nodelayer}
		\node [style=white with black border] (1) at (1, 0) {};
		\node [style=white with black border] (2) at (-0.5, 1) {};
		\node [style=white with black border] (3) at (-0.5, -1) {};
		\node [style=none] (32) at (6, -1) {};
		\node [style=white with black border] (36) at (3.5, -2) {};
		\node [style=white with black border] (37) at (5, -3) {};
		\node [style=white with black border] (38) at (5, -1) {};
		\node [style=black simple style] (46) at (2, 0) {};
		\node [style=white with black border] (47) at (3, 0) {};
		\node [style=white with black border] (48) at (4, 0) {};
		\node [style=white with black border] (49) at (5, 0) {};
		\node [style=white with black border] (50) at (0, 0) {};
		\node [style=white with black border] (51) at (4.5, -2) {};
		\node [style=white with black border] (53) at (1.5, -2) {};
		\node [style=white with black border] (54) at (0.5, -2) {};
		\node [style=white with black border] (55) at (-0.5, -2) {};
		\node [style=black simple style] (56) at (2.5, -2) {};
		\node [style=none] (80) at (8, -1) {};
		\node [style=white with black border] (81) at (10.5, 0) {};
		\node [style=white with black border] (82) at (9, 1) {};
		\node [style=white with black border] (83) at (9, -1) {};
		\node [style=white with black border] (84) at (12, -2) {};
		\node [style=white with black border] (85) at (13.5, -3) {};
		\node [style=white with black border] (86) at (13.5, -1) {};
		\node [style=white with black border] (88) at (11.5, 0) {};
		\node [style=white with black border] (89) at (12.5, 0) {};
		\node [style=white with black border] (90) at (13.5, 0) {};
		\node [style=white with black border] (91) at (9.5, 0) {};
		\node [style=white with black border] (92) at (13, -2) {};
		\node [style=white with black border] (93) at (11, -2) {};
		\node [style=white with black border] (94) at (10, -2) {};
		\node [style=white with black border] (95) at (9, -2) {};
	\end{pgfonlayer}
	\begin{pgfonlayer}{edgelayer}
		\draw (1) to (46);
		\draw (47) to (48);
		\draw (48) to (49);
		\draw (36) to (56);
		\draw (56) to (53);
		\draw (53) to (54);
		\draw (54) to (55);
		\draw (50) to (1);
		\draw (50) to (2);
		\draw (3) to (50);
		\draw (46) to (47);
		\draw (36) to (51);
		\draw (51) to (38);
		\draw (51) to (37);
		\draw [style=arrow] (32.center) to (80.center);
		\draw (88) to (89);
		\draw (89) to (90);
		\draw (93) to (94);
		\draw (94) to (95);
		\draw (91) to (81);
		\draw (91) to (82);
		\draw (83) to (91);
		\draw (84) to (92);
		\draw (92) to (86);
		\draw (92) to (85);
	\end{pgfonlayer}
\end{tikzpicture}
\end{center}
\vspace{5pt}

If we denote by $\{E_O, E_{12}\}$ the exceptional divisors corresponding to the singular points of $X$ and by $\{\hat{T}_1,\hat{T}_2\}$ the tropes, then, $E_O$ and $E_{12}$ intersect both $\hat{T}_1$ and $\hat{T}_2$, and the four points of intersection correspond to the four singular points of $Y$. The $D_4^0$ singularities correspond to the intersection of the tropes with the singular points with the greatest multiplicities in $X$ ($E_O\cap\hat{T}_{2}$ and $E_{12}\cap\hat{T}_{1}$) and the $A_3$ singularities correspond to the other two ($E_O\cap\hat{T}_{1}$ and $E_{12}\cap\hat{T}_{2}$).\\

\begin{center}
\begin{tikzpicture}
	\begin{pgfonlayer}{nodelayer}
		\node [style=none] (11) at (1.75, 2) {};
		\node [style=none] (12) at (3, 2.5) {};
		\node [style=none] (20) at (3, 0.75) {};
		\node [style=none] (21) at (1.75, 1.25) {};
		\node [style=none] (25) at (2.25, 0.75) {};
		\node [style=none] (26) at (2.25, 2.5) {};
		\node [style=none] (27) at (2.25, 2.75) {$E_O$};
		\node [style=none] (31) at (1.25, 2) {$\hat{T}_{1}$};
		\node [style=none] (32) at (1.25, 1.25) {$\hat{T}_{2}$};
		\node [style=none] (35) at (3, 2.75) {$E_{12}$};
		\node [style=none] (41) at (3.5, 1.25) {};
		\node [style=none] (43) at (3.5, 2) {};
	\end{pgfonlayer}
	\begin{pgfonlayer}{edgelayer}
		\draw (26.center) to (25.center);
		\draw (12.center) to (20.center);
		\draw (43.center) to (11.center);
		\draw (41.center) to (21.center);
	\end{pgfonlayer}
\end{tikzpicture}
\end{center}

It follows that the minimal resolution of the Kummer surface contains eighteen $(-2)$-curves which are the proper transforms of the four lines described above, and fourteen coming from the desingularisation of the $A_3$ and $D_4^0$ singularities. The intersection graph of these curves is given by the following diagram:

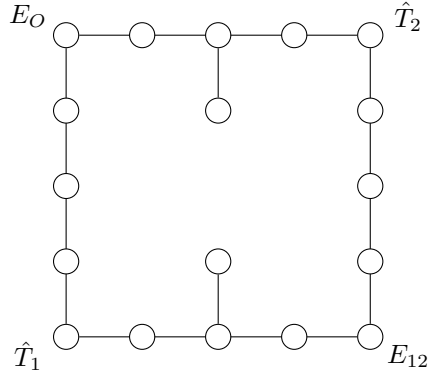
\begin{center}
\begin{tikzpicture}
	\begin{pgfonlayer}{nodelayer}
		\node [style=white with black border] (14) at (5, 1) {};
		\node [style=white with black border] (15) at (5, 2) {};
		\node [style=none] (21) at (0.5, 4.25) {$E_O$};
		\node [style=none] (24) at (0.5, -0.25) {$\hat{T}_1$};
		\node [style=none] (25) at (5.5, -0.25) {$E_{12}$};
		\node [style=none] (27) at (5.5, 4.25) {$\hat{T}_2$};
		\node [style=white with black border] (35) at (5, 0) {};
		\node [style=white with black border] (38) at (2, 4) {};
		\node [style=white with black border] (40) at (1, 4) {};
		\node [style=white with black border] (43) at (1, 2) {};
		\node [style=white with black border] (44) at (1, 1) {};
		\node [style=white with black border] (45) at (1, 3) {};
		\node [style=white with black border] (47) at (2, 0) {};
		\node [style=white with black border] (49) at (1, 0) {};
		\node [style=white with black border] (51) at (3, 1) {};
		\node [style=white with black border] (53) at (5, 3) {};
		\node [style=white with black border] (60) at (3, 0) {};
		\node [style=white with black border] (64) at (3, 3) {};
		\node [style=white with black border] (66) at (4, 0) {};
		\node [style=none] (70) at (3, 4) {};
		\node [style=white with black border] (71) at (3, 4) {};
		\node [style=white with black border] (72) at (4, 4) {};
		\node [style=white with black border] (73) at (5, 4) {};
	\end{pgfonlayer}
	\begin{pgfonlayer}{edgelayer}
		\draw (40) to (73);
		\draw (73) to (35);
		\draw (35) to (49);
		\draw (49) to (40);
		\draw (71) to (64);
		\draw (60) to (51);
	\end{pgfonlayer}
\end{tikzpicture}

\end{center}

A justification for why the curves intersect in this way will be provided in the next section.\\

As in the ordinary case, Katsura and Kondō proved that every Kummer surface associated to an almost ordinary abelian surface admits a model as a quartic in $\mathbb{P}^3_{x,y,z,t}$ of the form:
\begin{gather*}
b_3^2 c_1 x^4+b_2^2 d_1 y^4+b_1^2 d_1 z^4+b_4^2 c_1 t^4 \\
+\left(b_3^2 d_2+b_2^2 c_2+\left(a_1+a_2\right)^2 c_3+\left(a_1+a_2\right) b_2 b_3\right) x^2 y^2 \\
+\left(b_3^2 d_3+b_1^2 c_3+\left(a_1+a_2\right)^2 c_2+\left(a_1+a_2\right) b_1 b_3\right) x^2 z^2 \\
+\left(b_2^2 d_3+b_4^2 c_3+\left(a_1+a_2\right)^2 d_2+\left(a_1+a_2\right) b_2 b_4\right) y^2 t^2 \\
+\left(b_1^2 d_2+b_4^2 c_2+\left(a_1+a_2\right)^2 d_3+\left(a_1+a_2\right) b_1 b_4\right) z^2 t^2 \\
+\left(a_1+a_2\right)^2\left(b_3 x^2 y z+b_2 x y^2 t+b_1 x z^2 t+b_4 y z t^2\right)=0 .\\
\end{gather*}

Our model can be related to theirs by an explicit (but lengthy) change of coordinates, documented in the notebook \texttt{Part 3}. For the study of automorphisms, this is unnecessary: one can specialise directly from the ordinary model by setting $\alpha_{3}\mapsto \alpha_{2}$ and $\beta_{3}\mapsto\beta_{2}$.\\

Through a suitable change of coordinates, one may also relate the automorphisms of our model to the quartic model of Katsura and Kondō. In particular, this yields an explicit expression for the Cremona transformation, which they were unable to compute. It is given by
\begin{align*}
\phi([x:y:z:t])=[x':y':z':t']
\end{align*} 

where
\begin{align*}
x'&=\sqrt{d_1}x\big(\sqrt{b_2}y+\sqrt{b_1}z\big)^2, &
y'&=\sqrt{c_1}y\big(\sqrt{b_3}x+\sqrt{b_4}t\big)^2,\\
z'&=\sqrt{c_1}z\big(\sqrt{b_3}x+\sqrt{b_4}t\big)^2, &
t'&=\sqrt{d_1}t\big(\sqrt{b_2}y+\sqrt{b_1}z\big)^2.\\
\end{align*}

Specialising the partially desingularised model that we computed of $Y$, we can also connect this with the model in $\mathbb{P}^5$ of Katsura and Kondō.\\

\subsection{The geometry of Kummer surfaces in the supersingular case} \label{sssuper}
In the supersingular case the polynomial $g(x)$ has a single root over the splitting field of $g$, which we denote by $\alpha_{1}$. In particular, the corresponding Jacobian has no non-trivial $2$-torsion points. The point of the Kummer surface corresponding to the identity of the abelian surface is an elliptic singularity of type $\Circled{4}^{\,1}_{0,1}$, which in our model $X$ is given by the coordinates $[0:0:0:1]$.
Although no non-trivial $2$-torsion exists, there is still a Weierstrass point $w_{1}$ on $\calC$ corresponding to $(\alpha_{1},\beta_{1})$, and we define the associated trope $T_{1}$ to be the image of $\calC \times \{ w_{1} \}$.\\

Now consider the blow-up that was described in Section \ref{secblowup}. 
In the supersingular case, the singular point $\Circled{4}^1_{\,0,1}$ is a contraction of five lines in a tree configuration in which the central $(-2)$-curve has multiplicity two and the other four curves are three $(-2)$-curves and one $(-3)$-curve.
Then, the desingularisation map corresponds to the following transformation \cite[Theorem 6.3]{Schroer2009TheSurfaces}:
\begin{center}
\begin{tikzpicture}
	\begin{pgfonlayer}{nodelayer}
		\node [style=black with white font] (1) at (-7, 0) {$-3$};
		\node [style=white with black border] (3) at (-6, 1) {};
		\node [style=white with black border] (4) at (-6, 0) {};
		\node [style=white with black border] (5) at (-5, 0) {};
		\node [style=white with black border] (6) at (-6, -1) {};
		\node [style=none] (8) at (-2, 0) {};
		\node [style=none] (9) at (-4, 0) {};
		\node [style=white with black border] (11) at (-0.309017, -0.951057) {};
		\node [style=white with black border] (13) at (0.809017, 0.587785) {};
		\node [style=white with black border] (14) at (-0.309017, 0.951057) {};
		\node [style=none] (16) at (-5.75, 0.25) {$2$};
		\node [style=black with white font] (17) at (-1, 0) {$-3$};
		\node [style=black with white font] (18) at (0, 0) {$-3$};
		\node [style=white with black border] (19) at (0.809017, -0.587785) {};
	\end{pgfonlayer}
	\begin{pgfonlayer}{edgelayer}
		\draw (9.center) to (8.center);
		\draw [style=new edge style 0] (1) to (4);
		\draw [style=new edge style 0] (4) to (3);
		\draw [style=new edge style 0] (4) to (5);
		\draw [style=new edge style 0] (4) to (6);
		\draw [style=arrow] (9.center) to (8.center);
		\draw (17) to (18);
		\draw (18) to (14);
		\draw (18) to (13);
		\draw (18) to (19);
		\draw (18) to (11);
	\end{pgfonlayer}
\end{tikzpicture}
\end{center}
\vspace{5pt}

If we denote by $E_O$ the exceptional divisor corresponding to the singular point of $X$ and by $\hat{T}_1$ the trope,  then the singularity lies precisely in the intersection $E_{O} \cap \hat{T}_{1}$. As in the previous cases, there is a Cremona transformation in $Y$ exchanging $E_O$ and $\hat{T}_1$, and both it and its counterpart on $X$ admit explicit descriptions in our model.
\\

In the supersingular case, the desingularisation model obtained by Katsura and Kondō is given by the following equation:
\begin{gather*}
\left(b_3^2 c_1+b_7^2 c_3\right) x^4+\left(b_2^2 d_1+b_8^2 c_3\right) y^4+\left(b_1^2 d_1+b_6^2 d_3\right) z^4+\left(b_4^2 c_1+b_5^2 d_3\right) t^4 \\
+b_5\left(b_1 b_5+b_4 b_7\right) x t^3+b_7\left(b_2 b_7+b_3 b_5\right) x^3 t+b_2\left(b_2 b_6+b_3 b_8\right) x y^3+b_8\left(b_2 b_6+b_3 b_8\right) y^3 z \\
+b_3\left(b_2 b_7+b_3 b_5\right) x^3 y+b_4\left(b_1 b_5+b_4 b_7\right) z t^3+b_6\left(b_1 b_8+b_4 b_6\right) y z^3+b_1\left(b_1 b_8+b_4 b_6\right) z^3 t \\
+\left(b_2^2 c_2+b_3^2 d_2\right) x^2 y^2+\left(b_1^2 c_3+b_3^2 d_3+b_6^2 c_1+b_7^2 d_1\right) x^2 z^2+\left(b_5^2 c_2+b_7^2 d_2\right) x^2 t^2 \\
+\left(b_6^2 d_2+b_8^2 c_2\right) y^2 z^2+\left(b_2^2 d_3+b_4^2 c_3+b_5^2 d_1+b_8^2 c_1\right) y^2 t^2+\left(b_1^2 d_2+b_4^2 c_2\right) z^2 t^2 \\
+b_7\left(b_2 b_6+b_3 b_8\right) x^2 y z+b_3\left(b_1 b_5+b_4 b_7\right) x^2 z t+b_8\left(b_2 b_7+b_3 b_5\right) x y^2 t+b_2\left(b_1 b_8+b_4 b_6\right) y^2 z t \\
+b_1\left(b_2 b_6+b_3 b_8\right) x y z^2+b_6\left(b_1 b_5+b_4 b_7\right) x z^2 t+b_4\left(b_2 b_7+b_3 b_5\right) x y t^2+b_5\left(b_1 b_8+b_4 b_6\right) y z t^2=0 .\\
\end{gather*}

Our model is slightly simpler: it is obtained as a specialisation of the almost ordinary case by substituting $\alpha_{2} \mapsto \alpha_{1}$ and $\beta_{2} \mapsto \beta_{1}$ in the defining equation.\newpage
In the almost ordinary and ordinary cases, it is straightforward to relate our model to that of Katsura and Kondō: sending tropes to tropes and singular points to singular points almost determines the change of variables.  
In contrast, in the supersingular case there is only one trope and one singularity, and we were therefore unable to find a change of coordinates matching our model with Katsura and Kondō's.\\

\section{Weddle surfaces and the blow-ups of the exceptional lines} \label{sectionweddle}
Since they were first studied, one of the key features of quartic Kummer surfaces was the fact that, over algebraically closed fields, they were isomorphic to their projective dual. As a result, projecting away from a singular point gives rise to birationally equivalent quartic surfaces known as \textbf{Weddle surfaces}.\\

In characteristic zero, the construction of these surfaces is the following. As described in Section \ref{compsection}, given a model of a Kummer surface in $\mathbb{P}^3$ as a quartic surface with sixteen nodes, we can construct a blow-up as the intersection in $\mathbb{P}^5$ of three quadrics. As this blow-up is a birational map, we can construct an inverse map, which is well-defined outside of the singular locus of $X$. Furthermore, as this map only depends on the four first coordinates $b_1,b_2,b_3,b_4$, the projection map of the first four coordinates $\mathbb{P}^5\dashedarrow\mathbb{P}^3$, defines a map from $Y$ into $\mathbb{P}^3$, such that the closure of its image is given by a quartic surface $W\subset\mathbb{P}^3$.\\

After noticing that this map is well-defined outside of $b_1=b_2=b_3=b_4=0$, which is the exceptional line $E_O$, one can check that the Weddle surface geometrically corresponds to the map $\pi_O$ that consists of projecting $Y$ away from $E_O$. In characteristic not two, this transformation contracts the tropes $\hat{T}_1$, $\hat{T}_2$, $\hat{T}_3$, $\hat{T}_4$, $\hat{T}_5$ and $\hat{T}_6$, which are the ones meeting $E_O$, into six singular points of $W$ of type $A_1$, which we will denote by $Q_i$. The expression in coordinates for these points in our model are given by
\begin{align*}
Q_i=\Big[1:\omega_i:\omega_i^2:\frac{g(\omega_i)}{2}\Big].
\end{align*}

From this description, we deduce that, as the coordinates of the $Q_i$ depend exclusively of the  Weierstrass points, one can use these points to recover our initial curve $\mathcal{C}$ from the equation of the Weddle.\\

The images of the other exceptional lines $E_{ij}$ and tropes $\hat{T}_{ijk}$ are also lines in the Weddle surface, and they have a very nice geometric description \cite{Moore1928OnSurface}. All the singular points $Q_i$ are in general position meaning that no four of them lie in the same plane\footnote{ This can be seen from the description in coordinates of the singular points, as the matrix of the coordinates of any four points can be changed by a linear change of coordinates to a Vandermonde matrix, and therefore its determinant is never zero as all the $\omega_i$ are different.}, and if we consider the plane going through three of these singular points, say $Q_i$, $Q_j$ and $Q_k$, then the intersection of this plane with the Weddle surface always consists of the union of $\pi_O(E_{ij})$, $\pi_O(E_{ij})$, $\pi_O(E_{ij})$, and  $\pi_O(\hat{T}_{ijk})$. Furthermore, there is a special rational curve which we will denote by $C$ going through all the singular points, which is a twisted cubic defined by the equations:
\begin{align*}
b_2^2 - b_1 b_3=-2 b_2 b_4 + b_1 b_2 g_0 + b_2^2 g_1 + b_2 b_3 g_2 + 
 b_3^2 g_3=-2 b_1 b_4 + b_1^2 g_0 + b_1 b_2 g_1 + b_1 b_3 g_2 + b_2 b_3 g_3=0.
\end{align*}

We now consider the blow-up of the line $E_O$ in $Y$, which is closely related to the Weddle surface. As $Y$ is smooth and $E_O$ is a smooth subvariety of it, the blow-up is isomorphic to $Y$, so no new information is gained from this in characteristic zero. However, understanding the blow-up process will help us understand the blow-up of the exceptional lines in the specialisation in characteristic two.\\

The blow-up scheme of $E_O$, $\BlO$ is the Zariski closure of the image of the graph morphism
\begin{align*}
\Gamma_{\pi_O}:Y&\xdashrightarrow{\hspace{7mm}}Y\times\mathbb{P}^3\\
[b_1:b_2:b_3:b_4:b_5:b_6]&\longmapsto[b_1:b_2:b_3:b_4:b_5:b_6]\times[b_1:b_2:b_3:b_4].\\
\end{align*}
Let $\varphi_O:\BlO\rightarrow Y$ be the blow-up map. We can easily see that $\BlO\subseteq Y\times W$, and we can therefore describe the subvarieties of $\BlO$ as the restriction to $\BlO$ of subvarieties of $Y\times W$. Then, 
\begin{align*}
    \varphi_O^*(E_O)&=E_O\times C, &
    \varphi_O^*(E_{ij})&=E_{ij}\times \pi_O(E_{ij}),\\
    \varphi_O^*(\hat{T}_{i})&=\hat{T}_{i}\times Q_i, &
    \varphi_O^*(\hat{T}_{ijk})&=\hat{T}_{ijk}\times \pi_O(\hat{T}_{ijk}).\\
\end{align*}

Projecting away from any of the $32$ lines of $Y$ (the sixteen tropes or the sixteen exceptional lines) would also give us a map from $Y$ into a quartic surface in $\mathbb{P}^3$ with the same singularities. Any of these maps can be described as $\tau\circ \pi_O$, where $\tau$ is any automorphism of $Y$ exchanging the trope that we are projecting and $E_O$.\\

We will now see what happens when the field of definition has characteristic two, and describe the resulting singularities of the Weddle surface and what we obtain when we blow up the exceptional lines. It is worth mentioning that, in a recent article, Dolgachev \cite{Dolgachev2023K3Two} generalised the notion of Weddle surface for fields of characteristic two by defining them as the locus of singular points in the web of quadrics going through a set of six points of the form $Q_i$. The notion of Weddle surface we will refer to in the next subsections is different and corresponds to the surface obtained when we project away from the exceptional line $E_O$ in $Y$.\\

\subsection{The ordinary case}
The Weddle surface associated to an ordinary genus two curve has three $A_3$ singularities and four $A_1$ singularities.
Projecting away from $E_O$ blows up the singular points that are in the intersection of the tropes $\hat{T}_1$, $\hat{T}_2$, $\hat{T}_3$ with $E_O$ into three lines $L_1$, $L_2$ and $L_3$, and it contracts these tropes. As each of the tropes contains two singular points, these tropes are contracted to $A_3$ singularities. If we denote these singularities by $Q_i$, it is easy to check that the coordinates of these $Q_i$ in our model are given by
\begin{align*}
Q_i=[1:\alpha_i:\alpha_i^2:\beta_i].\\
\end{align*}
In addition to these three singular points, there are also four additional singular points of type $A_1$. One of them, which we will denote $Q_O$, has coordinates $Q_O=[0:0:0:1]$ and corresponds to the contraction of $E_O$ under $\pi_O$. The other three correspond to the images under the projection map $\pi_O$ of the three singularities of $Y$ that lie in $\hat{T}_{123}$, that is, they are in $\pi_O(E_{ij}\cap\hat{T}_{123})$. We will denote these by $Q_{ij}$.\\

Similarly to the characteristic zero case, we can recover both the coordinates of the Weierstrass points and the curve we started with from the singular points. All the singular points of the Weddle surface except for $Q_O$ lie in the same plane, which is given by the equation\\
\begin{adjustbox}{width=\textwidth}
\begin{minipage}{1.1\textwidth}
\begin{gather*}
(\alpha_2\alpha_3(\alpha_2+\alpha_3)\beta_1+\alpha_1\alpha_3(\alpha_1+\alpha_3)\beta_2+\alpha_2\alpha_3(\alpha_2+\alpha_3)\beta_1)\overline{b}_1+((\alpha_2+\alpha_3)^2\beta_1+(\alpha_1+\alpha_3)^2\beta_2+(\alpha_1+\alpha_2)^2\beta_3)\overline{b}_2\\
\hspace{50pt}+((\alpha_2+\alpha_3)\beta_1+(\alpha_1+\alpha_3)\beta_2+(\alpha_1+\alpha_2)\beta_3)\overline{b}_3+(\alpha_1+\alpha_2)(\alpha_1+\alpha_3)(\alpha_2+\alpha_3)\overline{b}_4=0.\hspace{50pt}\\
\end{gather*}
\end{minipage}
\end{adjustbox}
The intersection of this plane with the surface is the union of four lines corresponding to $\pi_O(\hat{T}_{123})$, $\pi_O(E_{12})$, $\pi_O(E_{13})$ and $\pi_O(E_{23})$ and is represented in the following diagram:\\[5pt]

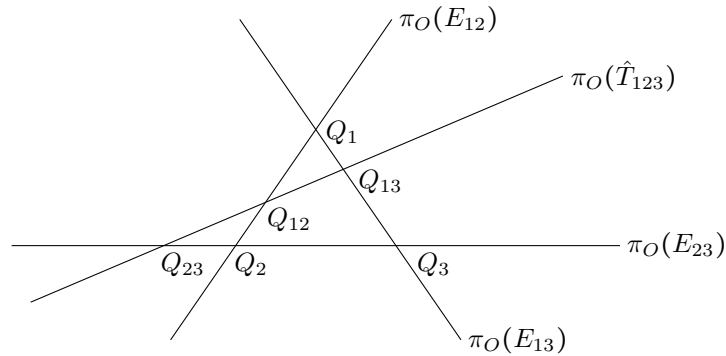
\begin{center}
\begin{tikzpicture}
	\begin{pgfonlayer}{nodelayer}
		\node [style=none] (0) at (-0.85, -0.25) {$Q_2$};
		\node [style=none] (1) at (1.55, -0.25) {$Q_3$};
		\node [style=none] (2) at (0.35, 1.5) {$Q_1$};
		\node [style=none] (4) at (-1.91, -1.25) {};
		\node [style=none] (5) at (1, 3) {};
		\node [style=none] (6) at (-1, 3) {};
		\node [style=none] (7) at (1.91, -1.25) {};
		\node [style=none] (8) at (-4, 0) {};
		\node [style=none] (9) at (4, 0) {};
		\node [style=none] (10) at (3.25, 2.25) {};
		\node [style=none] (11) at (-3.75, -0.75) {};
		\node [style=none] (12) at (1.75, 3) {$\pi_O(E_{12})$};
		\node [style=none] (13) at (4.05, 2.25) {$\pi_O(\hat{T}_{123})$};
		\node [style=none] (14) at (4.75, 0) {$\pi_O(E_{23})$};
		\node [style=none] (15) at (2.66, -1.25) {$\pi_O(E_{13})$};
		\node [style=none] (16) at (0.85, 0.85) {$Q_{13}$};
		\node [style=none] (17) at (-0.35, 0.35) {$Q_{12}$};
		\node [style=none] (18) at (-1.75, -0.25) {$Q_{23}$};
	\end{pgfonlayer}
	\begin{pgfonlayer}{edgelayer}
		\draw (6.center) to (7.center);
		\draw (4.center) to (5.center);
		\draw (8.center) to (9.center);
		\draw (10.center) to (11.center);
	\end{pgfonlayer}
\end{tikzpicture}
\end{center}

Now, consider the blow-up of the curve $E_O$, $\varphi_O:\BlO\rightarrow Y$, which is defined in the exact same way as for fields of characteristic not two. Then, the pullbacks of the exceptional lines and tropes of $Y$ are given by
\begin{align*}
    \varphi_O^*(E_O)&=E_O\times Q_O, &
    \varphi_O^*(E_{ij})&=E_{ij}\times \pi_O(E_{ij}),\\
    \varphi_O^*(\hat{T}_{i})&=\hat{T}_{i}\times Q_i, &
    \varphi_O^*(\hat{T}_{123})&=\hat{T}_{123}\times \pi_O(\hat{T}_{123}).\\
\end{align*}
Furthermore, as we are blowing up $E_O$, which contained three singular points of type $A_1$ corresponding to the intersection of $E_O$ with $\hat{T}_1$, $\hat{T}_2$ and $\hat{T}_3$, these singularities are resolved, and we have that if we let $L_{i}$ to be the line in $W$ going through $Q_O$ and $Q_i$,
\begin{align*}
\varphi_O^{-1}(E_O\cap\hat{T}_{i})=(E_O\cap\hat{T}_{i})\times L_{i}.\\
\end{align*}
None of the other singular points are resolved and we have that their preimages under the blow-up are
\begin{align*}
\varphi_O^{-1}(E_{ij}\cap\hat{T}_{i})&=(E_{ij}\cap\hat{T}_{i})\times Q_{i},\\
\varphi_O^{-1}(E_{ij}\cap\hat{T}_{123})&=(E_{ij}\cap\hat{T}_{123})\times Q_{ij}.\\
\end{align*}
Therefore, $\BlO$ has nine $A_1$ singularities. One can also understand $\BlO$ as a blow-up of the Weddle surface that resolves $Q_O$ and blows up the central exceptional curve of each of the $A_3$ singularities that we denoted $Q_i$. This is a diagram illustrating the blow-up process:\\

\begin{tikzpicture}
	\begin{pgfonlayer}{nodelayer}
		\node [style=none] (24) at (1, 1.35) {$E_O$};
		\node [style=none] (26) at (2.5, 2.15) {$\hat{T}_3$};
		\node [style=none] (28) at (1, -1.4) {$\hat{T}_1$};
		\node [style=none] (30) at (3, -1.4) {$E_{12}$};
		\node [style=none] (31) at (2.5, -0.65) {$E_{13}$};
		\node [style=none] (32) at (3, 1.4) {$\hat{T}_2$};
		\node [style=none] (33) at (4.5, 2.1) {$E_{23}$};
		\node [style=none] (34) at (4.5, -0.65) {$\hat{T}_{123}$};
		\node [style=white with black border] (35) at (3, 1) {};
		\node [style=white with black border] (36) at (3, -1) {};
		\node [style=white with black border] (38) at (4.5, 1.75) {};
		\node [style=white with black border] (40) at (2.5, 1.75) {};
		\node [style=white with black border] (43) at (1, 1) {};
		\node [style=white with black border] (46) at (1, -1) {};
		\node [style=white with black border] (48) at (2.5, -0.25) {};
		\node [style=white with black border] (56) at (4.5, -0.25) {};
		\node [style=none] (61) at (2.5, 0.95) {};
		\node [style=none] (62) at (2.5, 1.05) {};
		\node [style=none] (63) at (2.95, -0.25) {};
		\node [style=none] (64) at (3.05, -0.25) {};
		\node [style=A1 node] (129) at (3.5, 1.75) {};
		\node [style=A1 node] (131) at (2, 1) {};
		\node [style=A1 node] (132) at (1, 0) {};
		\node [style=A1 node] (135) at (3.5, 1.75) {};
		\node [style=A1 node] (150) at (3.5, -0.25) {};
		\node [style=A1 node] (155) at (2, -1) {};
		\node [style=A1 node] (160) at (3, 0) {};
		\node [style=A1 node] (165) at (4.5, 0.75) {};
		\node [style=A1 node] (171) at (1.75, -0.625) {};
		\node [style=A1 node] (172) at (3.75, -0.625) {};
		\node [style=A1 node] (173) at (3.75, 1.375) {};
		\node [style=A1 node] (174) at (1.75, 1.375) {};
		\node [style=A1 node] (175) at (2.5, 0.75) {};
		\node [style=none] (177) at (7.25, 1.75) {$L_3$};
		\node [style=none] (178) at (6.1, 0) {$L_1$};
		\node [style=none] (181) at (7.25, 0.75) {$L_2$};
		\node [style=A3 node] (184) at (8.5, 1) {};
		\node [style=white with black border] (185) at (8.5, -1) {};
		\node [style=white with black border] (186) at (10, 1.75) {};
		\node [style=A3 node] (187) at (8, 1.75) {};
		\node [style=A1 node] (188) at (6.5, 1) {};
		\node [style=A3 node] (189) at (6.5, -1) {};
		\node [style=white with black border] (190) at (8, -0.25) {};
		\node [style=white with black border] (191) at (10, -0.25) {};
		\node [style=none] (192) at (8, 0.95) {};
		\node [style=none] (194) at (8.45, -0.25) {};
		\node [style=none] (195) at (8.55, -0.25) {};
		\node [style=A1 node] (196) at (9, 1.75) {};
		\node [style=white with black border] (197) at (7.5, 1) {};
		\node [style=white with black border] (198) at (6.5, 0) {};
		\node [style=A3 node] (199) at (9, 1.75) {};
		\node [style=A1 node] (200) at (9, -0.25) {};
		\node [style=A3 node] (201) at (7.5, -1) {};
		\node [style=A3 node] (202) at (8.5, 0) {};
		\node [style=A1 node] (203) at (10, 0.75) {};
		\node [style=A3 node] (204) at (7.25, -0.625) {};
		\node [style=A1 node] (205) at (9.25, -0.625) {};
		\node [style=A3 node] (206) at (9.25, 1.375) {};
		\node [style=white with black border] (207) at (7.25, 1.375) {};
		\node [style=A3 node] (208) at (8, 0.75) {};
		\node [style=none] (209) at (7.75, 0.5) {};
		\node [style=none] (210) at (8.25, 0.5) {};
		\node [style=none] (211) at (7.75, 2) {};
		\node [style=none] (212) at (9.25, 2) {};
		\node [style=none] (213) at (9.25, 1.5) {};
		\node [style=none] (214) at (8.25, 1.5) {};
		\node [style=none] (215) at (8.25, 0.5) {};
		\node [style=none] (216) at (8.25, 1.5) {};
		\node [style=none] (217) at (6.25, -1.25) {};
		\node [style=none] (218) at (7.75, -1.25) {};
		\node [style=none] (219) at (7.75, -0.75) {};
		\node [style=none] (220) at (6.25, -0.8) {};
		\node [style=none] (221) at (7, -0.425) {};
		\node [style=none] (222) at (7.5, -0.425) {};
		\node [style=none] (223) at (7.5, -0.875) {};
		\node [style=none] (224) at (7, -0.875) {};
		\node [style=none] (225) at (7.25, -0.75) {};
		\node [style=none] (226) at (6.25, -1.25) {};
		\node [style=none] (227) at (8.25, -0.24) {};
		\node [style=none] (228) at (8.75, -0.24) {};
		\node [style=none] (229) at (8.25, 1.15) {};
		\node [style=none] (230) at (9.1, 1.575) {};
		\node [style=none] (231) at (9.5, 1.575) {};
		\node [style=none] (232) at (9.5, 1.225) {};
		\node [style=none] (233) at (8.75, 0.85) {};
		\node [style=none] (234) at (9, -1.4) {$\pi_O(E_{12})$};
		\node [style=none] (235) at (10.75, -0.5) {$\pi_O(\hat{T}_{123})$};
		\node [style=none] (236) at (8.5, 2.2) {\textcolor{darkestred}{$A_3$}};
		\node [style=none] (237) at (9, 0.75) {\textcolor{darkestred}{$A_3$}};
		\node [style=none] (238) at (7, -1.5) {\textcolor{darkestred}{$A_3$}};
		\node [style=none] (239) at (10.9, 1.75) {$\pi_O(E_{23})$};
		\node [style=white with black border] (248) at (13.9, 1) {};
		\node [style=white with black border] (249) at (13.9, -1) {};
		\node [style=white with black border] (250) at (15.4, 1.75) {};
		\node [style=white with black border] (251) at (13.4, 1.75) {};
		\node [style=white with black border] (252) at (11.9, 1) {};
		\node [style=white with black border] (253) at (11.9, -1) {};
		\node [style=white with black border] (254) at (13.4, -0.25) {};
		\node [style=white with black border] (255) at (15.4, -0.25) {};
		\node [style=none] (256) at (13.4, 0.95) {};
		\node [style=none] (257) at (13.4, 1.05) {};
		\node [style=none] (258) at (13.85, -0.25) {};
		\node [style=none] (259) at (13.95, -0.25) {};
		\node [style=A1 node] (260) at (14.4, 1.75) {};
		\node [style=white with black border] (261) at (12.9, 1) {};
		\node [style=white with black border] (262) at (11.9, 0) {};
		\node [style=A1 node] (263) at (14.4, 1.75) {};
		\node [style=A1 node] (264) at (14.4, -0.25) {};
		\node [style=A1 node] (265) at (12.9, -1) {};
		\node [style=A1 node] (266) at (13.9, 0) {};
		\node [style=A1 node] (267) at (15.4, 0.75) {};
		\node [style=A1 node] (268) at (12.65, -0.625) {};
		\node [style=A1 node] (269) at (14.65, -0.625) {};
		\node [style=A1 node] (270) at (14.65, 1.375) {};
		\node [style=white with black border] (271) at (12.65, 1.375) {};
		\node [style=A1 node] (272) at (13.4, 0.75) {};
		\node [style=none] (292) at (3.75, 2.25) {};
		\node [style=none] (293) at (9.25, 2.25) {};
		\node [style=none] (294) at (2.25, -1.5) {};
		\node [style=none] (295) at (13.25, -1.5) {};
		\node [style=none] (296) at (6.4, 3.25) {$\pi_O$};
		\node [style=none] (297) at (7.65, -2.1) {$\Gamma_{\pi_O}$};
	\end{pgfonlayer}
	\begin{pgfonlayer}{edgelayer}
		\draw (43) to (40);
		\draw (40) to (38);
		\draw (43) to (35);
		\draw (35) to (38);
		\draw (43) to (46);
		\draw (46) to (36);
		\draw (36) to (35);
		\draw (46) to (48);
		\draw (36) to (56);
		\draw (56) to (38);
		\draw (48) to (61.center);
		\draw (62.center) to (40);
		\draw (48) to (63.center);
		\draw (64.center) to (56);
		\draw (188) to (187);
		\draw (187) to (186);
		\draw (184) to (186);
		\draw (188) to (189);
		\draw (189) to (185);
		\draw (185) to (184);
		\draw (185) to (191);
		\draw (191) to (186);
		\draw (190) to (192.center);
		\draw (190) to (204);
		\draw [style=A3 edge white] (216.center)
			 to (213.center)
			 to (212.center)
			 to (211.center)
			 to (209.center)
			 to (215.center)
			 to cycle;
		\draw (188) to (184);
		\draw [style=A3 edge] (218.center)
			 to (226.center)
			 to (220.center)
			 to (221.center)
			 to (222.center)
			 to (223.center)
			 to (224.center)
			 to (225.center)
			 to (219.center)
			 to cycle;
		\draw [style=A3 edge white] (227.center)
			 to (229.center)
			 to (230.center)
			 to (231.center)
			 to (232.center)
			 to (233.center)
			 to (228.center)
			 to cycle;
		\draw (190) to (194.center);
		\draw (195.center) to (191);
		\draw (252) to (251);
		\draw (251) to (250);
		\draw (252) to (248);
		\draw (248) to (250);
		\draw (252) to (253);
		\draw (253) to (249);
		\draw (249) to (248);
		\draw (253) to (254);
		\draw (249) to (255);
		\draw (255) to (250);
		\draw (254) to (256.center);
		\draw (257.center) to (251);
		\draw (254) to (258.center);
		\draw (259.center) to (255);
		\draw [style=dashed arrow, bend right=15] (294.center) to (295.center);
		\draw [style=dashed arrow, in=150, out=30] (292.center) to (293.center);
	\end{pgfonlayer}
\end{tikzpicture}
\vspace{10pt}

An analogous argument applies if one blows up any of the other exceptional lines $E_{ij}$: the blow-up resolves precisely the three singular points lying on that line. This observation will later allow us to construct an explicit model for the resolution of $Y$.

\begin{proposition} \label{prop blowup}
Let $\mathcal{I}_{ij}=\langle\mu^{(ij)}_{1},\mu^{(ij)}_{2},\mu^{(ij)}_{3},\mu^{(ij)}_{4} \rangle$ be the ideal generated by four linear polynomials on the variables $\{\overline{b}_1,\dots,\overline{b}_6\}$ such that $E_{ij}=\mathbb{V}(\mathcal{I}_{ij})$ in $Y$. Let $\pi_{ij}$ be the morphism
\begin{align*}
\pi_{ij}:Y&\xdashrightarrow{\hspace{7mm}} \mathbb{P}^3\\
[\overline{b}_1:\overline{b}_2:\overline{b}_3:\overline{b}_4:\overline{b}_5:\overline{b}_6]&\longmapsto[\mu^{(ij)}_{1}:\mu^{(ij)}_{2}:\mu^{(ij)}_{3}:\mu^{(ij)}_{4}].
\end{align*}

Then, the Zariski closure of the image of the graph morphism $\Gamma_{\pi_{ij}}$ corresponds to the blow-up scheme $\mathrm{Bl}_{E_{ij}}(Y)$ along the subvariety $E_{ij}$, which blows up the three singular points in $E_{ij}$.\\

Furthermore, let $Z=E_{O}\cup E_{12}\cup E_{13}\cup E_{23}$ and consider the birational map $\phi:Y\dashedarrow (\mathbb{P}^3)^4$, which acts in each copy of $\mathbb{P}^3$ as $\pi_O,\pi_{12},\pi_{13}$ and $\pi_{23}$ respectively. Then, the Zariski closure of the image of the graph morphism $\Gamma_\phi$ is the blow-up scheme $\mathrm{Bl}_Z(Y)$ and this is a resolution of the twelve $A_1$ singularities of $Y$.
\end{proposition}
\begin{proof}[Proof of Proposition \ref{prop blowup}]
As described in Subsection \ref{ssord}, the group $(\mathbb{Z}/2\mathbb{Z})^3$ acts linearly on $Y$. In particular, for every $E_{ij}$, there is a linear action $\tau_{ij}$ on $Y$ of order two interchanging $E_O$ and $E_{ij}$.\newpage As the image with respect of $\tau_{ij}$ of the ideal $\langle \overline{b}_1,\overline{b}_2,\overline{b}_3,\overline{b}_4\rangle$ is $\mathcal{I}_{ij}$, we deduce that $\tau_{ij}$ induces a linear isomorphism between $\mathrm{Bl}_O(\tau_{ij}(Y))$ and $\mathrm{Bl}_{E_{ij}}(Y)$, showing that the construction of $\Gamma_{\pi_{ij}}$ corresponds to a blow-up of the exceptional line $E_{ij}$.\\

As for the second half of the statement, $\phi$ is a birational map that has a well-defined inverse away from $Z$. As the four exceptional lines are disjoint, in the open $Y\setminus (Z\setminus E_{ij})$, we have that the following diagram is commutative
$$
\begin{tikzcd}
                                                                      & Y\setminus(Z\setminus E_{ij}) \arrow[ld, "\phi"'] \arrow[rd, "\pi_{ij}"] &                                                 \\
(\mathbb{P}^3)^4 \arrow[rr, "pr"] &                                                                              & \mathbb{P}^3
\end{tikzcd}
$$
and the projection map $pr$ is an isomorphism between $\im(\phi)$ and $\im(\pi_{ij})$. As a consequence, we deduce that $\Gamma_{\phi}(Y\setminus(Z\setminus E_{ij}))\cong\Gamma_{\pi_{ij}}(Y\setminus(Z\setminus E_{ij}))$ and that in an open set not containing $Z\setminus E_{ij}$, $\Gamma_\phi^{-1}$ is a blow-up of each of the $E_{ij}$. As these lines are all disjoint, we deduce that $\Gamma_\phi^{-1}$ blows up the union of all of the lines, and as all twelve singularities of $Y$ lie in $Z$, and we have seen that the blow-up of each line resolves three of them, we deduce that $\mathrm{Bl}_Z(Y)$ resolves the twelve $A_1$ singularities. 
\end{proof}

We can relate the geometry in characteristic zero and characteristic two as follows. Let $Y$ be defined over a discretely valued field $K$ with perfect residue field $k$ of characteristic two, and suppose $Y$ has good ordinary reduction. Writing $\overline{Y}$ for the reduction over $k$, the blow-up $\BlO$ specialises to $\mathrm{Bl}_{\overline{E}_O}(\overline{Y})$. Over $\overline{K}$ we can find four exceptional lines on $Y$ which reduce to $\overline{E}_O$, forming a $D_4$ sublattice inside $\Pic(Y)$. Likewise, there are three further sets of four lines, each set reducing to $E_{12}$, $E_{13}$, and $E_{23}$ respectively. The Galois group $\Gal(\overline{K}/K)$ acts on the exceptional divisors of $Y$, and this action descends to an action of $\Gal(\overline{k}/k)$ on the set $\{E_O, E_{12}, E_{13}, E_{23}\}$. This indicates that for a Kummer surface to have good reduction at $2$, the action of $\Gal(\overline{K}/K)$ on the $2$-torsion of $\Jacc$ must be compatible with the induced action on the $2$-torsion of its reduction $\overline{\calJ}$. We will expand on this point in Section~\ref{section good red}.\\

\subsection{The almost ordinary case}
The Weddle surface in the almost ordinary case has one $A_3$, one $A_7$ and one $D_5^0$ singularity. In this case, projecting away from $E_O$ contracts the tropes that meet $E_O$, which are $\hat{T}_1$ and $\hat{T}_2$, into two singularities $Q_1$ and $Q_2$ of types $A_7$ and $D_5$ respectively, whose coordinates in the notation of Subsection \ref{almost ord section} are given by
\begin{align*}
Q_i=[1:\alpha_i:\alpha_i^2:\beta_i].\\
\end{align*}
From a computation of the Tjurina number of $Q_2$, we deduce that this singular point has to be of type $D_5^0$. The remaining singularity $Q_O$, which has coordinates $Q_O=[0:0:0:1]$, is of the type $A_3$ as it is a contraction of $E_O$ and two other lines. Similarly to the ordinary case, all the singularities lie in the same plane, which in this case is given by the equation
\begin{align*}
\alpha_1\alpha_2 \overline{b}_1+(\alpha_1+\alpha_2 )\overline{b}_2+\overline{b}_3=0.\\
\end{align*}
The intersection of this plane with the Weddle surface is three lines intersecting the three singular points. The line that has multiplicity two also happens to be the image under the projection from $Y$ to the Weddle surface of the line $E_{12}$:
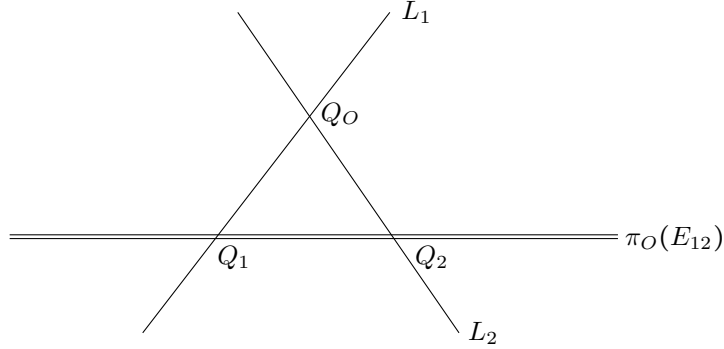
\begin{center}
\begin{tikzpicture}
	\begin{pgfonlayer}{nodelayer}
		\node [style=none] (0) at (-1.05, -0.25) {$Q_1$};
		\node [style=none] (1) at (1.55, -0.25) {$Q_2$};
		\node [style=none] (2) at (0.35, 1.65) {$Q_O$};
		\node [style=none] (4) at (-2.25, -1.25) {};
		\node [style=none] (5) at (1, 3) {};
		\node [style=none] (6) at (-1, 3) {};
		\node [style=none] (7) at (1.91, -1.25) {};
		\node [style=none] (8) at (-4, 0) {};
		\node [style=none] (9) at (4, 0) {};
		\node [style=none] (12) at (1.35, 3) {$L_1$};
		\node [style=none] (14) at (4.75, 0) {$\pi_O(E_{12})$};
		\node [style=none] (15) at (2.23, -1.25) {$L_2$};
		\node [style=none] (17) at (-4, 0.05) {};
		\node [style=none] (18) at (4, 0.05) {};
	\end{pgfonlayer}
	\begin{pgfonlayer}{edgelayer}
		\draw (6.center) to (7.center);
		\draw (4.center) to (5.center);
		\draw (8.center) to (9.center);
		\draw (17.center) to (18.center);
	\end{pgfonlayer}
\end{tikzpicture}
\end{center}
Consider now the blow-up of the curve $E_O$, $\varphi_O:\BlO\rightarrow Y$. Then, the pullbacks of the exceptional lines and tropes of $Y$ are given by
\begin{align*}
    \varphi_O^*(E_O)&=E_O\times Q_O, &
    \varphi_O^*(E_{12})&=E_{12}\times \pi_O(E_{12}),\\
    \varphi_O^*(\hat{T}_{1})&=\hat{T}_{1}\times Q_1, &
    \varphi_O^*(\hat{T}_{2})&=\hat{T}_{2}\times Q_2.\\
\end{align*}
Since we are blowing-up $E_O$, which contained a singular point of type $A_3$ corresponding to the intersection of $E_O$ with $\hat{T}_1$, one of the exceptional curves gets blown-up, so that 
\begin{align*}
\varphi_O^{-1}(E_O\cap\hat{T}_{1})=(E_O\cap\hat{T}_{1})\times L_1,
\end{align*}
and the $A_3$ singularity becomes an $A_2$ singularity in the point $(E_O\cap\hat{T}_1)\times Q_O$. Likewise, the singular point of type $D_4^1$ corresponding to the intersection of $E_O$ with $\hat{T}_2$ gets blown up into the line
\begin{align*}
\varphi_O^{-1}(E_O\cap\hat{T}_{2})=(E_O\cap\hat{T}_{2})\times L_2,
\end{align*}
and the $D_4^1$ singularity becomes an $A_3$ singularity in the point $(E_O\cap\hat{T}_1)\times Q_2$. None of the other singular points are resolved and we have that their preimages under the blow-up are
\begin{align*}
\varphi_O^{-1}(E_{12}\cap\hat{T}_{1})&=(E_{12}\cap\hat{T}_{1})\times Q_{1},\\
\varphi_O^{-1}(E_{12}\cap\hat{T}_{2})&=(E_{12}\cap\hat{T}_{2})\times Q_{2}.\\
\end{align*}
Therefore, $\BlO$ has one $A_2$, two $A_3$ and one $D_4^1$ singularity. The blow-up process is described by the following diagram:

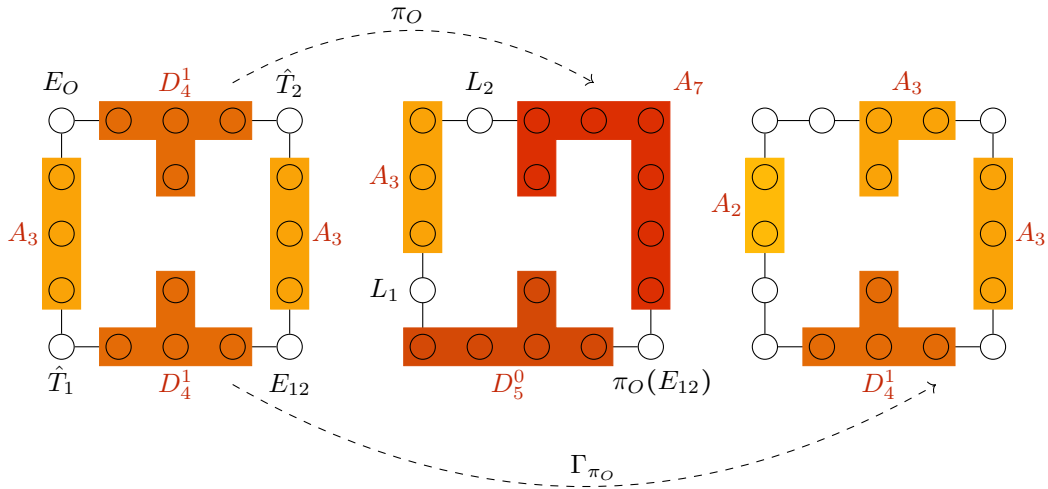
\begin{center}
\begin{tikzpicture}
	\begin{pgfonlayer}{nodelayer}
		\node [style=none] (21) at (1, 4.5) {$E_O$};
		\node [style=none] (24) at (1, 0.56) {$\hat{T}_1$};
		\node [style=none] (25) at (4, 0.5) {$E_{12}$};
		\node [style=none] (27) at (4, 4.5) {$\hat{T}_2$};
		\node [style=white with black border] (35) at (4, 1) {};
		\node [style=white with black border] (40) at (1, 4) {};
		\node [style=white with black border] (49) at (1, 1) {};
		\node [style=white with black border] (73) at (4, 4) {};
		\node [style=none] (93) at (8.9, 0.53) {$\pi_O(E_{12})$};
		\node [style=white with black border] (95) at (8.75, 1) {};
		\node [style=white with black border] (96) at (6.5, 4) {};
		\node [style=white with black border] (99) at (5.75, 1.75) {};
		\node [style=white with black border] (125) at (13.25, 1) {};
		\node [style=white with black border] (126) at (11, 4) {};
		\node [style=white with black border] (127) at (10.25, 4) {};
		\node [style=white with black border] (129) at (10.25, 1.75) {};
		\node [style=white with black border] (132) at (10.25, 1) {};
		\node [style=white with black border] (140) at (13.25, 4) {};
		\node [style=D4 node] (149) at (1.75, 4) {};
		\node [style=D4 node] (150) at (3.25, 4) {};
		\node [style=D4 node] (151) at (2.5, 4) {};
		\node [style=D4 node] (152) at (2.5, 3.25) {};
		\node [style=A3 node] (153) at (1, 3.25) {};
		\node [style=A3 node] (154) at (1, 2.5) {};
		\node [style=A3 node] (155) at (1, 1.75) {};
		\node [style=A3 node] (156) at (4, 3.25) {};
		\node [style=A3 node] (157) at (4, 2.5) {};
		\node [style=A3 node] (158) at (4, 1.75) {};
		\node [style=D4 node] (159) at (1.75, 1) {};
		\node [style=D4 node] (160) at (2.5, 1) {};
		\node [style=D4 node] (161) at (3.25, 1) {};
		\node [style=D4 node] (162) at (2.5, 1.75) {};
		\node [style=none] (163) at (1.5, 4.25) {};
		\node [style=none] (164) at (1.5, 3.75) {};
		\node [style=none] (165) at (2.25, 3.75) {};
		\node [style=none] (166) at (2.25, 3) {};
		\node [style=none] (167) at (2.75, 3) {};
		\node [style=none] (168) at (2.75, 3.75) {};
		\node [style=none] (169) at (3.5, 3.75) {};
		\node [style=none] (170) at (3.5, 4.25) {};
		\node [style=none] (171) at (3.75, 3.5) {};
		\node [style=none] (172) at (4.25, 3.5) {};
		\node [style=none] (173) at (3.75, 1.5) {};
		\node [style=none] (174) at (4.25, 1.5) {};
		\node [style=none] (175) at (0.75, 3.5) {};
		\node [style=none] (176) at (1.25, 3.5) {};
		\node [style=none] (177) at (1.25, 1.5) {};
		\node [style=none] (178) at (0.75, 1.5) {};
		\node [style=none] (179) at (1, 3.5) {};
		\node [style=none] (180) at (1.5, 1.25) {};
		\node [style=none] (181) at (1.5, 0.75) {};
		\node [style=none] (182) at (3.5, 0.75) {};
		\node [style=none] (183) at (3.5, 1.25) {};
		\node [style=none] (184) at (2.75, 1.25) {};
		\node [style=none] (185) at (2.25, 1.25) {};
		\node [style=none] (186) at (2.25, 2) {};
		\node [style=none] (187) at (2.75, 2) {};
		\node [style=A7 node] (188) at (8.75, 1.75) {};
		\node [style=A7 node] (189) at (8.75, 2.5) {};
		\node [style=A7 node] (190) at (8.75, 3.25) {};
		\node [style=A7 node] (191) at (8.75, 4) {};
		\node [style=A7 node] (192) at (8, 4) {};
		\node [style=A7 node] (193) at (7.25, 4) {};
		\node [style=A7 node] (194) at (7.25, 3.25) {};
		\node [style=D5 node] (195) at (5.75, 1) {};
		\node [style=D5 node] (196) at (6.5, 1) {};
		\node [style=D5 node] (197) at (7.25, 1) {};
		\node [style=D5 node] (198) at (7.25, 1.75) {};
		\node [style=D5 node] (199) at (8, 1) {};
		\node [style=A3 node] (200) at (5.75, 2.5) {};
		\node [style=A3 node] (201) at (5.75, 3.25) {};
		\node [style=A3 node] (202) at (5.75, 4) {};
		\node [style=none] (203) at (5.5, 4.25) {};
		\node [style=none] (204) at (6, 4.25) {};
		\node [style=none] (205) at (5.5, 2.25) {};
		\node [style=none] (206) at (6, 2.25) {};
		\node [style=none] (207) at (5.5, 1.25) {};
		\node [style=none] (208) at (5.5, 0.75) {};
		\node [style=none] (209) at (7, 1.25) {};
		\node [style=none] (210) at (7, 2) {};
		\node [style=none] (211) at (7.5, 2) {};
		\node [style=none] (212) at (7.5, 1.25) {};
		\node [style=none] (213) at (8.25, 1.25) {};
		\node [style=none] (214) at (8.25, 0.75) {};
		\node [style=none] (215) at (8.5, 1.5) {};
		\node [style=none] (216) at (9, 1.5) {};
		\node [style=none] (217) at (8.5, 3.75) {};
		\node [style=none] (218) at (7.5, 3.75) {};
		\node [style=none] (219) at (7.5, 3) {};
		\node [style=none] (220) at (7, 3) {};
		\node [style=none] (221) at (7, 4.25) {};
		\node [style=none] (222) at (9, 4.25) {};
		\node [style=A2 node] (223) at (10.25, 3.25) {};
		\node [style=A2 node] (224) at (10.25, 2.5) {};
		\node [style=A3 node] (225) at (11.75, 4) {};
		\node [style=A3 node] (226) at (12.5, 4) {};
		\node [style=A3 node] (227) at (11.75, 3.25) {};
		\node [style=A3 node] (228) at (13.25, 3.25) {};
		\node [style=A3 node] (229) at (13.25, 2.5) {};
		\node [style=A3 node] (230) at (13.25, 1.75) {};
		\node [style=D4 node] (231) at (11, 1) {};
		\node [style=D4 node] (232) at (11.75, 1) {};
		\node [style=D4 node] (233) at (12.5, 1) {};
		\node [style=D4 node] (234) at (11.75, 1.75) {};
		\node [style=none] (235) at (11.5, 4.25) {};
		\node [style=none] (236) at (11.5, 3) {};
		\node [style=none] (237) at (12, 3) {};
		\node [style=none] (238) at (12, 3.75) {};
		\node [style=none] (239) at (12.75, 3.75) {};
		\node [style=none] (240) at (12.75, 4.25) {};
		\node [style=none] (241) at (13, 3.5) {};
		\node [style=none] (242) at (13.5, 3.5) {};
		\node [style=none] (243) at (13.5, 1.5) {};
		\node [style=none] (244) at (13, 1.5) {};
		\node [style=none] (245) at (12.75, 1.25) {};
		\node [style=none] (246) at (12.75, 0.75) {};
		\node [style=none] (247) at (10.75, 0.75) {};
		\node [style=none] (248) at (10.75, 1.25) {};
		\node [style=none] (249) at (11.5, 1.25) {};
		\node [style=none] (250) at (12, 1.25) {};
		\node [style=none] (251) at (11.5, 2) {};
		\node [style=none] (252) at (12, 2) {};
		\node [style=none] (253) at (10, 2.25) {};
		\node [style=none] (254) at (10.5, 2.25) {};
		\node [style=none] (255) at (10.5, 3.5) {};
		\node [style=none] (256) at (10, 3.5) {};
		\node [style=none] (261) at (2.5, 4.5) {\textcolor{darkestred}{$D_4^1$}};
		\node [style=none] (262) at (3.25, 4.5) {};
		\node [style=none] (263) at (8, 4.5) {};
		\node [style=none] (264) at (3.25, 0.5) {};
		\node [style=none] (265) at (12.5, 0.5) {};
		\node [style=none] (266) at (4.5, 2.5) {\textcolor{darkestred}{$A_3$}};
		\node [style=none] (267) at (2.5, 0.5) {\textcolor{darkestred}{$D_4^1$}};
		\node [style=none] (268) at (0.5, 2.5) {\textcolor{darkestred}{$A_3$}};
		\node [style=none] (269) at (5.55, 5.4) {$\pi_O$};
		\node [style=none] (270) at (8, -0.6) {$\Gamma_{\pi_O}$};
		\node [style=none] (271) at (11.75, 0.5) {\textcolor{darkestred}{$D_4^1$}};
		\node [style=none] (272) at (6.875, 0.5) {\textcolor{darkestred}{$D_5^0$}};
		\node [style=none] (273) at (5.25, 3.25) {\textcolor{darkestred}{$A_3$}};
		\node [style=none] (274) at (9.25, 4.5) {\textcolor{darkestred}{$A_7$}};
		\node [style=none] (275) at (9.75, 2.875) {\textcolor{darkestred}{$A_2$}};
		\node [style=none] (276) at (13.75, 2.5) {\textcolor{darkestred}{$A_3$}};
		\node [style=none] (277) at (12.125, 4.5) {\textcolor{darkestred}{$A_3$}};
		\node [style=none] (278) at (6.5, 4.5) {$L_{2}$};
		\node [style=none] (279) at (5.25, 1.75) {$L_{1}$};
	\end{pgfonlayer}
	\begin{pgfonlayer}{edgelayer}
		\draw (40) to (73);
		\draw (73) to (35);
		\draw (35) to (49);
		\draw (49) to (40);
		\draw (127) to (140);
		\draw (140) to (125);
		\draw (125) to (132);
		\draw (132) to (127);
		\draw (151) to (152);
		\draw (162) to (160);
		\draw [style=D4 edge] (166.center)
			 to (165.center)
			 to (164.center)
			 to (163.center)
			 to (170.center)
			 to (169.center)
			 to (168.center)
			 to (167.center)
			 to cycle;
		\draw [style=A3 edge] (171.center)
			 to (173.center)
			 to (174.center)
			 to (172.center)
			 to cycle;
		\draw [style=A3 edge] (177.center)
			 to (178.center)
			 to (175.center)
			 to (176.center)
			 to cycle;
		\draw [style=D4 edge] (182.center)
			 to (183.center)
			 to (184.center)
			 to (187.center)
			 to (186.center)
			 to (185.center)
			 to (180.center)
			 to (181.center)
			 to cycle;
		\draw (191) to (95);
		\draw (193) to (194);
		\draw (195) to (95);
		\draw (197) to (198);
		\draw (202) to (195);
		\draw (202) to (191);
		\draw [style=A7 edge] (215.center)
			 to (216.center)
			 to (222.center)
			 to (221.center)
			 to (220.center)
			 to (219.center)
			 to (218.center)
			 to (217.center)
			 to cycle;
		\draw [style=A3 edge] (204.center)
			 to (206.center)
			 to (205.center)
			 to (203.center)
			 to cycle;
		\draw [style=D5 edge] (210.center)
			 to (209.center)
			 to (207.center)
			 to (208.center)
			 to (214.center)
			 to (213.center)
			 to (212.center)
			 to (211.center)
			 to cycle;
		\draw (225) to (227);
		\draw (234) to (232);
		\draw [style=A2 edge] (255.center)
			 to (254.center)
			 to (253.center)
			 to (256.center)
			 to cycle;
		\draw [style=A3 edge] (235.center)
			 to (236.center)
			 to (237.center)
			 to (238.center)
			 to (239.center)
			 to (240.center)
			 to cycle;
		\draw [style=A3 edge] (242.center)
			 to (241.center)
			 to (244.center)
			 to (243.center)
			 to cycle;
		\draw [style=D4 edge] (250.center)
			 to (252.center)
			 to (251.center)
			 to (249.center)
			 to (248.center)
			 to (247.center)
			 to (246.center)
			 to (245.center)
			 to cycle;
		\draw [style=dashed arrow, bend right] (264.center) to (265.center);
		\draw [style=dashed arrow, in=150, out=30] (262.center) to (263.center);
	\end{pgfonlayer}
\end{tikzpicture}
\end{center}

The action of $(\Z/2\Z)^2$ on $Y$ allows us to map $E_{12}$, $\hat{T}_1$ or $\hat{T}_2$ to $E_O$, so blowing-up any of those lines will produce the same configuration of singularities as blowing-up $E_O$.
Replicating the proof of Proposition \ref{prop blowup}, if we consider $\pi_{12}$ to be the map $Y\dashedarrow\mathbb{P}^3$ whose image is the four linear polynomials on $\{\overline{b}_1,\dots,\overline{b}_6\}$ defining the equations of $E_{12}$, we can construct a morphism $\phi:Y\dashedarrow(\mathbb{P}^3)^2$ such that the Zariski closure of the image of $\Gamma_\phi$ is the blow-up scheme $\mathrm{Bl}_{E_O\cup E_{12}}(Y)$. The singular points of $\mathrm{Bl}_{E_O\cup E_{12}}(Y)$ are then two $A_2$ and two $A_3$ singularities. Therefore, in the almost ordinary case it does not suffice to blow-up all the exceptional lines on $Y$ to obtain a smooth model.\\

\subsection{The supersingular case} This case is completely different than the previous two. Projecting away from $E_O$, we no longer get isolated singularities, but instead, $W$ has a singular line $L$ which is defined by the equation:
\begin{align*}
\alpha_1\overline{b}_1+\overline{b}_2=\alpha_1^2\overline{b}_1+\overline{b}_3=0.\\
\end{align*}
The trope $\hat{T}_1$ then gets contracted to the point
\begin{align*}
Q_1=[1:\alpha_1:\alpha_1^2:\beta_1].\\
\end{align*}
Finally, $\BlO$ blows up the singular point $P$ of $Y$ into a singular line which corresponds to $P\times L$. \\

\section{Kummer surfaces with everywhere good reduction over a quadratic field} \label{section good red}
Let $F$ be a number field and $v$ a non-Archimedean place of $F$ such that $K=F_v$ is a complete, discretely valued field with ring of integers $\calO_{K}$ and residue field $k$. A variety $X/F$ is said to have \textbf{good reduction at} $\boldsymbol{v}$ if there exists a scheme or algebraic space $\calX$ smooth and proper over $\calO_{K}$ with generic fibre $\calX_K\cong X$. We will say that $X/F$ has \textbf{potentially good reduction at} $v$, if there exists a finite field extension $L/F$ such that for all places $w$ lying above $v$, $X/L$ has good reduction at $w$.
A variety $X/F$ is said to have\textbf{ everywhere good reduction} if it has good reduction at every non-Archimedean place.\\

As stated in the introduction, we are interested in studying whether there exist K3 surfaces $X/F$ with everywhere good reduction. We will see that this is indeed the case, as we can find a scheme model of Kummer surface with everywhere good reduction, as a consequence of the following.\\

Let $A$ be an abelian surface over a number field $F$ and let $v$ be an non-Archimedean place.
\begin{itemize}
    \item If $v$ does not lie above two, then the Kummer surface associated to $A$ has good reduction at $v$ if and only if there exists a quadratic twist $A^\chi$ of $A$ such that $A^\chi$ has good reduction. This is a consequence of the work of Matsumoto \cite{Matsumoto2015OnSurfaces} and Overkamp \cite{Overkamp2021DegenerationSurfaces}.

    \item If $v$ lies above two, then the Kummer surface associated to $A$ has potentially good reduction if $A$ has good reduction at $v$. This is a consequence of Lazda and Skorobogatov \cite{Lazda2023ReductionCase} in the ordinary and almost ordinary case and Matsumoto \cite{Matsumoto2023Supersingular2} in the supersingular case.\\  
\end{itemize}

Starting with an abelian surface with everywhere good reduction, these results show that over possibly a field extension, its associated Kummer surface has everywhere good reduction. The goal of this section is to show that it is possible to explicitly construct an example of a Kummer surface with everywhere good reduction over a quadratic number field. 

\begin{theorem} \label{ordinaryexth}
Let $F=\mathbb{Q}(\sqrt{353})$, let $\omega=\frac{1+\sqrt{353}}{2}$ and let 
\begin{align*}
\mathcal{C}: y^2+g(x)y=f(x)
\end{align*}
where
\begin{align*}
g(x)&=(\omega + 1)x^3 + x^2 + \omega x + 1,\\
f(x)&=(-15\omega + 149)x^6 - (1119\omega  + 9948)x^5 - (36545\omega  + 325409)x^4\\
&- (363632\omega  + 5659370)x^3 - (622714\omega  + 5538975)x^2\\
&- (3284000\omega  + 288867915)x - 70532813\omega  - 627353458.
\end{align*}
Then, the Kummer surface associated to $\Jac(\calC)$ has everywhere good reduction over $F$. 
\end{theorem}
\begin{proof}[Proof of Theorem \ref{ordinaryexth}]
This curve was found by Dembélé \cite[Theorem 6.2]{Dembele2021OnReduction}. One can check that the discriminant of $\mathcal{C}$ is $-\epsilon^4$, where $\epsilon$ is the fundamental unit of $F$, $\Jac(\mathcal{C})$ has everywhere good reduction. By the previously mentioned results, its associated Kummer surface has good reduction at all non-Archimedean places not lying above two. Therefore, we only need to prove that the Kummer surface also has good reduction at the places lying above two. 
In order to do that, we will apply a criterion developed by Lazda and Skorobogatov \cite[Theorem 2]{Lazda2023ReductionCase}.\\

Let $A=\Jac(\mathcal{C})$ be an abelian surface with good (not supersingular) reduction at two, let $K$ be a discretely valued field with perfect residue field $k$ of characteristic two, and let $\calA/\calO_K$ be the Néron model of $A/K$, which is an abelian scheme with generic fiber $\calA_K\cong A$.\newpage
Let us fix an algebraic closure $\overline{K}$ of $K$, with residue field $\overline{k}$, and let $\Gamma_K$ denote the Galois group of $\overline{K}/K$.
Then, we have the exact sequence of $\Gamma_K$-modules:
\begin{align} \label{seq}
0\longrightarrow \calA[2]^\circ(\overline{K})\longrightarrow \calA[2](\overline{K})\longrightarrow\calA[2](\overline{k})\longrightarrow0
\end{align}
where $\calA[2]^\circ$ is the connected component of the identity of the $2$-torsion subscheme $\calA[2]\subseteq\calA$.
\begin{theorem}[\cite{Lazda2023ReductionCase}] \label{LSth} 
If $A$ has ordinary reduction, the Kummer surface associated to $A$ has good reduction over $K$ if and only if the exact sequence \eqref{seq} of $\Gamma_K$-modules splits. If $A$ has almost ordinary reduction, the Kummer surface associated to $A$ has good reduction over $K$ if and only if the $\Gamma_K$-module $\calA[2](\overline{K})$ is trivial. Moreover, in both cases the Kummer surface has good reduction with a scheme model.\\
\end{theorem}

As the curve $\mathcal{C}$ has ordinary reduction at two, we will apply the first part of the theorem.
Let $K$ be the completion of $F=\mathbb{Q}(\sqrt{353})$ at two. As $353$ is $1$ modulo $8$, we can easily check that $353$ is a square in $\mathbb{Q}_2$, and so, $K=\mathbb{Q}_2$. Then, $\calO_K=\mathbb{Z}_2$ and we deduce that $k=\mathbb{F}_2$. Furthermore, by computing the $2$-adic expansion, we can see that $\omega$ reduces to zero modulo two and therefore the reduction of $\mathcal{C}$ modulo two can be shown to have the equation
\begin{align*}
 y^2+(x^3 + x^2 + 1)y=x^6+x^2+x.\\
\end{align*}
As explained in Section \ref{PS}, the decomposition of $g(x)$ over $k$ determines the number of $2$-torsion points defined over $k$. As in this case $g(x)$ is irreducible over $\mathbb{F}_2$, $\calA[2](k)$ is trivial and $\calA[2](\ell)=(\mathbb{Z}/2\mathbb{Z})^2$ if and only if $\ell\supseteq \mathbb{F}_8=\mathbb{F}_2(\overline{\gamma})$, where $\overline{\gamma}^3+\overline{\gamma}^2+1=0$. The $2$-torsion points are of the form $\{\overline{P}_O,\overline{P}_{12},\overline{P}_{13},\overline{P}_{23}\}$ (as described in Section \ref{PS}) where we take $\alpha_1=\overline{\gamma}$, $\alpha_2=\overline{\gamma}^2$ and $\alpha_3=\overline{\gamma}^2+\overline{\gamma}+1$.\\

Hence, as a $\Gamma_K$-module, $\calA[2](\overline{k})$ admits only a cyclic action of order three, corresponding to the Frobenius action on $\mathbb{F}_8$ permuting its non-trivial elements. The number of $2$-torsion points defined over $K$ is determined by the factorisation of $f(x)+\tfrac{1}{4}g(x)^2$ into irreducible polynomials over $K$, and a direct computation in Magma confirms that
\begin{align*}
f(x)+\tfrac{1}{4}g(x)^2=\tfrac{1}{4}q_1(x)q_2(x),
\end{align*}
where $q_1$ and $q_2$ are the following irreducible polynomials over $\mathbb{Q}_2$
\begin{align*}
q_1(x)&=x^3 + (2088841801 + O(2^{32}))x^2 + (1097586240 + O(2^{32}))x + 553607353 + O(2^{32}),\\
q_2(x)&=x^3 + (1373013921 + O(2^{32}))x^2 - (1548938988 + O(2^{32}))x - 856394843 + O(2^{32}).\\
\end{align*}
In fact, this decomposition arises from the factorisation over $F$:
\begin{align*}
    f(x)+\tfrac{1}{4}g(x)^2=-\tfrac{3}{4}(19\omega+169)q_1(x)q_2(x)
\end{align*} where
\begin{align*}
q_1(x)&=x^3 + \tfrac{1}{3}(12\omega  - 5)x^2 + \tfrac{1}{12}(11\omega  + 5640)x + \tfrac{1}{12}(2507\omega  - 588),\\
q_2(x)&=x^3 + (4\omega  + 1)x^2 + (8\omega  + 468)x + 211\omega  + 365.\\
\end{align*}

As $f(x)+\tfrac{1}{4}g(x)^2$ decomposes into two cubic polynomials, $\lvert \calA[2](K)\rvert=1$ and since $\calA[2](K)\neq\calA[2](\overline{K})$, we deduce that there are elements of $\Gamma_K$ acting non-trivially on $\calA[2](\overline{K})$. Let $L$ be the unique unramified extension of degree three of $\mathbb{Q}_2$ which, without any loss of generality, we can consider it to be $\mathbb{Q}_2(\gamma)$ where $(\omega  + 1)\gamma^3 + \gamma^2 + \omega \gamma + \omega  + 1=0$. Then, over $L$, we have that
\begin{align*}
f(x)+\tfrac{1}{4}g(x)^2=\tfrac{1}{4}h_1(x)h_2(x)h_3(x)h_4(x)h_5(x)h_6(x),
\end{align*}
Here,
\begin{align*}
h_1(x)&=x-406904280\gamma^2 + 435522127\gamma - 1230442616 + O(2^{32}),\\
h_2(x)&=x + 394057577\gamma^2 - 1606502354\gamma + 490223466 + O(2^{32}),\\
h_3(x)&=x-1060895121\gamma^2 - 976503421\gamma + 681577303 + O(2^{32}),\\
h_4(x)&=x + 1307484884\gamma^2 + 1755128143\gamma - 56114964 + O(2^{32}),\\
h_5(x)&=x + 914512901\gamma^2 + 842339586\gamma - 1344868422 + O(2^{32}),\\
h_6(x)&=x-1148255961\gamma^2 - 449984081\gamma + 626513659 + O(2^{32}),
\end{align*}
and $q_1(x)=h_1(x)h_2(x)h_3(x)$ and $q_2(x)=h_4(x)h_5(x)h_6(x)$. Let $r_i$ denote the root of $h_i$, and let $P_{ij}$ be the $2$-torsion point associated to $r_i$ and $r_j$. 
As the polynomial completely splits over $L$, $\calA[2](L)=\calA[2](\overline{L})$ and, therefore, $\calA[2](\overline{L})$ is trivial as a $\Gamma_L$-module. We can therefore check that the only non-trivial actions of $\Gamma_K$ in $\calA[2](K)$ are the ones induced by $\Gal(L/K)\cong C_3$ which permute the roots of $q_1$ and $q_2$.\\

As $L$ is the maximal unramified extension of degree three of $\mathbb{Q}_2$, the action of $\Gamma_K$ on $\calA[2](\overline{k})$ is also by the group $C_3$ and it acts in a way that is compatible with the action on $\calA[2](\overline{k})$. More precisely, let $\varsigma\in S_6$ given in the cycle notation by $\varsigma=(123)(456)$, and let $\tau_\varsigma$ be the action of $\Gamma_K$ induced in $\calA[2](\overline{K})$ by $\tau_\varsigma(P_{ij})=P_{\varsigma(i)\varsigma(j)}$. Then, $\tau_\varsigma$ acts on $\calA[2](\overline{k})$ by permuting cyclically the roots of $g(x)$ and the short exact sequence 
\begin{align*}
0\longrightarrow \calA[2]^\circ(\overline{K})\longrightarrow \calA[2](\overline{K})\stackrel{f}{\longrightarrow}\calA[2](\overline{k})\longrightarrow0 
\end{align*}
splits as we can easily construct sections of it, for instance, by defining
\begin{align*}
\sigma(P)=\begin{cases}
P_O\,\quad\text{ if }P=\overline{P}_O\\
P_{12}\quad \text{ if }P=\overline{P}_{12}\\
P_{13}\quad \text{ if }P=\overline{P}_{13}\\
P_{23}\quad \text{ if }P=\overline{P}_{23}\\
\end{cases}
\end{align*}
as $\langle P_{12},P_{13}\rangle=(\Z/2\Z)^2\subset \calA[2](\overline{K})$. It can be checked that $r_1$, $r_2$ and $r_3$ reduce to $\alpha_1$, $\alpha_2$ and $\alpha_3$ respectively so $f\circ\sigma=id$.\end{proof}

Here is where we can draw a connection with Section \ref{sectionweddle}. Choosing a section of the short exact sequence \ref{seq} corresponds to selecting a set of four $2$-torsion points $\{ P_O, P_{12}, P_{13}, P_{23} \}$ with the same Galois action as the $2$-torsion over the residue field. In the model of the Kummer surface as an intersection of three quadrics in $\mathbb{P}^5$, the exceptional lines $\{ E_O, E_{12}, E_{13}, E_{23} \}$ are defined over the same extension of $\mathbb{Q}_2$ as their corresponding torsion points, and their union is defined over $\mathbb{Z}_2$, since the ideal defining this variety depends only on the coefficients of $q_1$.\\

Because the Galois action over $K$ is compatible with the reduction, the ideal of the union of these four lines reduces to the ideal of the corresponding four exceptional lines over $\mathbb{F}_2$. Blowing up these four lines on $Y$ thus yields a smooth model of the Kummer surface over $\mathbb{Z}_2$, whose reduction is the blow-up of the four exceptional lines over $\mathbb{F}_2$, resolving all twelve singular points, as previously observed.\\

In this example, we did not need to take any field extension to obtain good reduction of the Kummer surface at two. This is not generally the case, as we can see when we analyse the other examples in the articles, where we only obtain potential good reduction at the primes above two and we need to take field extensions to achieve good reduction.\\

The following table presents examples of curves $\mathcal{C}$ with ordinary reduction at two and everywhere good reduction over the field $\mathbb{Q}(\omega)$. The first six examples are taken from Dembélé and Kumar \cite{Dembele2016ExamplesReduction}, while the last two are from Dembélé \cite{Dembele2021OnReduction}. The final column indicates the degree of the minimal extension of $\mathbb{Q}_2(\omega)$ over which $\Kum(\mathcal{C})$ attains good reduction at two.
All the computations can be found in the file \texttt{Everywhere good reduction.m}.\\

\begin{table}[h]
\begin{adjustbox}{width=1\textwidth}
\begin{tabular}{|
>{\columncolor[HTML]{FFFFFF}}l 
>{\columncolor[HTML]{FFFFFF}}l |lc|}
\hline
\multicolumn{1}{|c}{\cellcolor[HTML]{FFFFFF}$g(x)$}                  & \multicolumn{1}{c}{\cellcolor[HTML]{FFFFFF}$f(x)$}      & \multicolumn{1}{|c}{\cellcolor[HTML]{FFFFFF}$\omega$}               & \cellcolor[HTML]{FFFFFF}$d$ \\ \hline
\cellcolor[HTML]{FFFFFF}                                              & $-4 x^6+(\omega-17) x^5+(12 \omega-27) x^4+(5 \omega-122) x^3$          & \multicolumn{1}{c}{}                                          &                             \\
\multirow{-2}{*}{\cellcolor[HTML]{FFFFFF}$\omega x^3+\omega x^2+\omega+1$}             & $+(45 \omega-25) x^2+(-9 \omega-137) x+14 \omega+9$                     & \multicolumn{1}{c}{\multirow{-2}{*}{$\frac{1+\sqrt{53}}{2}$}} & \multirow{-2}{*}{$2$}       \\ \hline
\cellcolor[HTML]{FFFFFF}                                              & $(\omega-5) x^6+(3 \omega-14) x^5+(3 \omega-19) x^4$                    & \multicolumn{1}{c}{}                                          &                             \\
\multirow{-2}{*}{\cellcolor[HTML]{FFFFFF}$x^3+x+1$}                   & $+(4 \omega-3) x^3-(3 \omega+16) x^2+(3 \omega+11) x-(\omega+4)$             & \multicolumn{1}{c}{\multirow{-2}{*}{$\frac{1+\sqrt{73}}{2}$}} & \multirow{-2}{*}{$4$}       \\ \hline
\cellcolor[HTML]{FFFFFF}                                              & $-2(4414 \omega+43089) x^6+(31147 \omega+303963) x^5$              &                                                                &                             \\
\cellcolor[HTML]{FFFFFF}                                              & $-10(4522 \omega+44133) x^4+2(17290 \omega+168687) x^3$            &                                                                &                             \\
\multirow{-3}{*}{\cellcolor[HTML]{FFFFFF}$\omega\left(x^3+1\right)$}       & $-18(816 \omega+7967) x^2+27(122 \omega+1189) x-(304 \omega+3003)$      & \multirow{-3}{*}{$\frac{1+\sqrt{421}}{2}$}                     & \multirow{-3}{*}{$2$}       \\ \hline
\cellcolor[HTML]{FFFFFF}                                              & $-2 x^6+(-3 \omega+1) x^5-219 x^4+(-83 \omega+41) x^3-1806 x^2$    &                                                                &                             \\
\multirow{-2}{*}{\cellcolor[HTML]{FFFFFF}$x^3+x^2+1$}                 & $+(-204 \omega+102) x-977$                                    & \multirow{-2}{*}{$\frac{1+\sqrt{409}}{2}$}                     & \multirow{-2}{*}{$4$}       \\ \hline
\cellcolor[HTML]{FFFFFF}                                              & $-134 x^6-(146 \omega-73) x^5-13427 x^4-(3255 \omega-1627) x^3$    &                                                                &                             \\
\multirow{-2}{*}{\cellcolor[HTML]{FFFFFF}$x^3+x+1$}                   & $-89746 x^2-(6523 \omega-3261) x-39941$                       & \multirow{-2}{*}{$\frac{1+\sqrt{809}}{2}$}                     & \multirow{-2}{*}{$4$}       \\ \hline
\cellcolor[HTML]{FFFFFF}                                              & $23 x^6+(90 \omega-45) x^5+33601 x^4+(28707 \omega-14354) x^3$     &                                                                &                             \\
\multirow{-2}{*}{\cellcolor[HTML]{FFFFFF}$x^3+x+1$}                   & $+3192149 x^2+(811953 \omega-405977) x+19904990$              & \multirow{-2}{*}{$\frac{1+\sqrt{929}}{2}$}                     & \multirow{-2}{*}{$4$}       \\ \hline
\cellcolor[HTML]{FFFFFF}                                              & $(13\omega + 77)x^6 + (503\omega + 6772)x^5 + (1504\omega + 131460)x^4$ &                                                                &                             \\
\cellcolor[HTML]{FFFFFF}                                              & $+ (16882\omega + 1727293)x^3 + (116734\omega + 10787410)x^2$      &                                                                &                             \\
\multirow{-3}{*}{\cellcolor[HTML]{FFFFFF}$\omega x^3 + x^2 + (\omega + 1)x + 1$} & $+ (398570\omega + 40121781)x + 611123\omega + 58505073$           & \multirow{-3}{*}{$\frac{1+\sqrt{421}}{2}$}                     & \multirow{-3}{*}{$4$} \\ \hline
\cellcolor[HTML]{FFFFFF}                                              & $(14154412\omega +275745514)x^6-(489014393\omega+9526607332)x^5$   &                                                                &                             \\
\cellcolor[HTML]{FFFFFF}                                              & $+ (7039395048\omega+137136152764)x^4- 54043428224\omega x^3$      &                                                                &                             \\
\cellcolor[HTML]{FFFFFF}                                              & $-1052833060832x^3+(233382395752\omega +4546578743807)x^2$    &                                                                &                             \\
\cellcolor[HTML]{FFFFFF}                                              & $- (537510739916\omega +10471376373574)x +515810377784\omega$                     &                                                                &                             \\
\multirow{-5}{*}{\cellcolor[HTML]{FFFFFF}$x^3+ \omega x^2+(\omega+1)x+\omega+1$}      & $+10048626384323$                          & \multirow{-5}{*}{$\frac{1+\sqrt{1597}}{2}$}                    & \multirow{-5}{*}{$4$} \\ \hline
\end{tabular}
\end{adjustbox}
\vspace{10pt}
\caption{Examples of curves with everywhere good reduction and ordinary reduction at $2$}
\end{table}

To understand why, in some examples, a field extension is required to obtain good reduction at two, let us consider, for instance, the third example in the table.  \\

Here, $K=\mathbb{Q}_2(\sqrt{421})$, $\calO_K=\mathbb{Z}_2[\omega]$ and as the minimal polynomial of $\omega$ is $x^2-x-105$, which is irreducible modulo two, we deduce that $k=\mathbb{F}_2(\overline{\omega})=\mathbb{F}_4$. Then, the reduction of $\mathcal{C}$ modulo two can be shown to have the equation
\begin{align*}
y^2+\overline{\omega}(x^3+1)=(1+\overline{\omega})x^5+x+1.
\end{align*}
Therefore, $g(x)$ completely splits over $k$
\begin{align*}
g(x)=\overline{\omega}(x^3+1)=(\overline{\omega}x+1)(x+1)(x+\overline{\omega}),
\end{align*}
and as $\calA[2](\overline{k})=\calA[2](k)$, we deduce that $\calA[2](\overline{k})$
is trivial as a $\Gamma_K$-module. However, 
\begin{align*}
f(x)+\tfrac{1}{4}g(x)^2=\tfrac{1}{4}h_1(x)h_2(x)q_3(x)q_4(x)
\end{align*}
where the following are all irreducible polynomials over $K$
\begin{align*}
h_1(x)&=x + 1312351119 -2028179001\omega + O(2^{32}),\\
h_2(x)&=x - 1300818437 -1345357737\omega  + O(2^{32}),\\
q_3(x)&=x^2 + (1256541238 + 188416644 \omega + O(2^{32}))\,x + (1294873809 -1495287772\omega + O(2^{32})),\\
q_4(x)&=x^2 + (-1426178004 - 209135522 \omega+ O(2^{32}))\,x + (- 1663860799 +724531893\omega + O(2^{32})),
\end{align*}
This implies that $\calA[2](\overline{K})$ is not trivial as a $\Gamma_K$-module, as there are non-trivial $K$-automorphisms acting on the Weierstrass points, and therefore the $2$-torsion. For instance, we have an action of order two permuting the two roots of $q_3$.\\

We can check that the only submodule of $\calA[2](\overline{K})$ that is trivial as a $\Gamma_K$-module is $\calA[2](K)$, which, as a group is isomorphic to $(\Z/2\Z)^2$ by what we have described in Section \ref{sectionsp}. However, this submodule is precisely $\calA[2]^\circ(\overline{K})$, and the image of this group in $\calA[2](\overline{k})$ is trivial. As a consequence, one cannot find a section of the exact sequence \eqref{seq}, as the image of any section would have to be trivial as a $\Gamma_K$-module, and we deduce that $\Kum(A)$ does not have good reduction over $\mathbb{Q}_2(\sqrt{421})$.\newpage Nevertheless, if we consider the ramified extension $L=\mathbb{Q}_2(\sqrt{421},i)$, then, over that extension, the polynomial $f(x)+\tfrac{1}{4}g(x)^2$ completely splits. Thus, $\calA[2](\overline{L})$ becomes trivial as a $\Gamma_L$-module, and we can easily construct sections as in the previous example. Through a similar reasoning, we can argue in the other seven examples which field extensions we need to take, and what their degrees are.\\

In the first example, $f(x)+\tfrac{1}{4}g(x)^2$ decomposes into the product of a quadratic and a quartic polynomial over $K=\mathbb{Q}_2(\omega)$ and the splitting field has Galois group $C_2^3$. Over the residue field $k=\mathbb{F}_4$, $g(x)$ decomposes into a linear and a quadratic factor; therefore the action of $\Gamma_K$ on $\calA[2](\overline{K})$ is by the group $C_2$. We checked that there are two possible quadratic extensions of $K$ compatible with the Galois action over which the sequence \eqref{seq} splits, namely, the two ramified extensions that split the quartic factor of $f(x)+\tfrac{1}{4}g(x)^2$. Each of these gives rise to eight possible sections that would split the sequence, so that in total over the splitting field we would have the sixteen possible sections that we described earlier.\\

In the rest of the examples, we always have that $f(x)+\tfrac{1}{4}g(x)^2$ is irreducible over $K$ and the splitting field has $A_4$ as its Galois group. Furthermore, over the residue field, $g(x)$ is also irreducible, so its Galois group is $C_3$. From the Sylow theorems, we deduce that there are four Sylow $3$-subgroups, which have index four in $A_4$, and from the Galois correspondence, we deduce that these must correspond to four field extensions of $K$ of degree four. Over any of these extensions, $f(x)+\tfrac{1}{4}g(x)^2$ splits into two cubic polynomials and we can construct four sections splitting the sequence \eqref{seq}. \\

\subsection{Kummer surfaces with everywhere good reduction and almost ordinary reduction at two}

A natural question that remains is whether one can construct a Kummer surface with everywhere good reduction and almost ordinary reduction at two. The answer is affirmative; however, no examples are currently known where the good reduction occurs over a quadratic number field. As observed in Theorem \ref{LSth}, good reduction at two is obtained over the field $K$ for which the sequence \eqref{seq} becomes trivial as a $\Gamma_K$-module. By the same reasoning as before, this field extension $K$ must coincide with the splitting field of $f(x)+\tfrac{1}{4}g(x)^2$.\\

There are only two examples in Dembélé's and Kumar article of abelian surfaces with everywhere good reduction that have good almost ordinary reduction at two:

\begin{table}[h]
\centering
\begin{tabular}{|
>{\columncolor[HTML]{FFFFFF}}l 
>{\columncolor[HTML]{FFFFFF}}l |cc|}
\hline
\multicolumn{1}{|c}{\cellcolor[HTML]{FFFFFF}$g(x)$} & \multicolumn{1}{c|}{\cellcolor[HTML]{FFFFFF}$f(x)$} & \cellcolor[HTML]{FFFFFF}$\omega$                & \cellcolor[HTML]{FFFFFF}$d$ \\ \hline
\cellcolor[HTML]{FFFFFF}                             & $2x^6+(-2 \omega+7)x^5+(-5 \omega+47) x^4+(-12 \omega+85)x^3$      &                                            &                             \\
\multirow{-2}{*}{\cellcolor[HTML]{FFFFFF}$-x- \omega$}    & $+(-13 \omega+97) x^2+(-8 \omega+56)x-2w+1$                   & \multirow{-2}{*}{$\frac{1+\sqrt{193}}{2}$} & \multirow{-2}{*}{$12$}      \\ \hline
\cellcolor[HTML]{FFFFFF}                             & $-2x^6-(2 \omega-1)x^5-45x^4-4(2 \omega-1)x^3-31x^2$          &                                            &                             \\
\multirow{-2}{*}{\cellcolor[HTML]{FFFFFF}$x+1$}      & $+ (\omega-1)x+9$                                        & \multirow{-2}{*}{$\frac{1+\sqrt{233}}{2}$} & \multirow{-2}{*}{$12$}      \\ \hline
\end{tabular}
\vspace{10pt}
\caption{Examples of curves with everywhere good reduction and almost ordinary reduction at $2$}
\end{table}

In both of this cases, we can check that the minimal extension over which $f(x)+\tfrac{1}{4}g(x)^2$ completely splits is the degree twelve extension $\mathbb{Q}_2(\sqrt{5}, \sqrt[3]{1+i})/\mathbb{Q}_2$, whose Galois group is the dihedral group of order twelve. The calculations are available in the file \texttt{Everywhere good reduction.m} as well.\\

One can easily check that this field extension decomposes in a degree two unramified part $\mathbb{Q}_2(\sqrt{5})/\mathbb{Q}_2$, and a degree six completely ramified part given by $\mathbb{Q}_2(\sqrt[3]{1+i})/\mathbb{Q}_2$. Therefore, if we set $K=\mathbb{Q}_2(\sqrt[3]{1+i})$, we find that the Jacobians of any of the two previous examples are abelian surfaces with good, almost ordinary reduction at two, such that $\calA[2](\overline{K})$ are unramified but non-trivial as a $\Gamma_K$-module.\\

Regarding other possible examples, Dąbrowski and Sadek \cite{Dabrowski2021GenusFields} computed a family of genus two curves with everywhere good reduction and almost ordinary reduction at two. We computed $400$ examples of their family and checked that in none of them, the associated Kummer surface has good reduction over the base field.\\
 
Finally, the rest of the examples of abelian surfaces with everywhere good reduction over a quadratic field have supersingular reduction at two. By Matsumoto's result \cite[Theorem 1.2]{Matsumoto2023Supersingular2}, there exists a field extension over which the Kummer surface acquires good reduction at two; however, his theorem does not provide a method to explicitly determine this extension in concrete examples.

\bibliographystyle{alpha}
\bibliography{references.bib}

\end{document}